\documentclass[11pt,table]{amsart}
\usepackage[margin=1.1in]{geometry}
\usepackage{amscd,amsmath,amsxtra,amsthm,amssymb,stmaryrd,xr,mathrsfs,mathtools,enumerate,commath, comment, orcidlink}
\usepackage{tikz}
\usetikzlibrary{arrows.meta,positioning}

\usetikzlibrary{calc}
\usepackage{tikz-3dplot}
\usepackage{stmaryrd}
\usepackage{multirow}
\usepackage{xcolor}
\usepackage{svg}
\usepackage{graphics}
\usepackage{longtable} 
\usepackage{pdflscape} 
\usepackage{booktabs}
\usepackage{hyperref}
\definecolor{vegasgold}{rgb}{0.77, 0.7, 0.35}
\definecolor{darkgoldenrod}{rgb}{0.72, 0.53, 0.04}
\definecolor{gold(metallic)}{rgb}{0.83, 0.69, 0.22}
\hypersetup{
 colorlinks=true,
 linkcolor=darkgoldenrod,
 filecolor=brown,      
 urlcolor=gold(metallic),
 citecolor=darkgoldenrod,
 }
\newtheorem{lthm}{Theorem}

\usepackage[all,cmtip]{xy}

\DeclareFontFamily{U}{wncy}{}
\DeclareFontShape{U}{wncy}{m}{n}{<->wncyr10}{}
\DeclareSymbolFont{mcy}{U}{wncy}{m}{n}
\DeclareMathSymbol{\Sh}{\mathord}{mcy}{"58}
\usepackage[T2A,T1]{fontenc}

\newtheorem{theorem}{Theorem}[section]
\newtheorem{lemma}[theorem]{Lemma}

\newtheorem*{theorem*}{Theorem}
\newtheorem*{ass*}{Assumption}
\newtheorem{definition}[theorem]{Definition}
\newtheorem{corollary}[theorem]{Corollary}
\newtheorem{remark}[theorem]{Remark}

\newtheorem{proposition}[theorem]{Proposition}

\newcommand{\Gal}{\op{Gal}}

\newcommand{\Z}{\mathbb{Z}}
\newcommand{\Q}{\mathbb{Q}}
\newcommand{\F}{\mathbb{F}}

\newcommand{\cO}{\mathcal{O}}

\newcommand{\op}[1]{\operatorname{#1}}

 \DeclareMathSymbol{\sha}{\mathord}{mcy}{"58}
 \makeatletter
\newcommand{\mylabel}[2]{#2\def\@currentlabel{#2}\label{#1}}
\makeatother

\numberwithin{equation}{section}

\title[Cyclotomic norm congruences]
{Cyclotomic norm congruences for zeta values and modular $L$-values}

\author[S.~Bhattacharya]{Shubhrajit Bhattacharya}
\address[Bhattacharya]{Department of Mathematics, The University of Chicago,
Eckhart Hall, 5734 S University Ave, Chicago, IL 60637, USA}
\email{shubhrajit@uchicago.edu}
\author[A.~Ray]{Anwesh Ray \orcidlink{0000-0001-6946-1559}}
\address[Ray]{Chennai Mathematical Institute, H1, SIPCOT IT Park,
Kelambakkam, Siruseri, Tamil Nadu 603103, India}
\email{anwesh@cmi.ac.in}

\author[R.~Sujatha]{R.~Sujatha \orcidlink{0000-0003-1221-0710}}
\address[Sujatha]{Department of Mathematics, University of British Columbia,
Vancouver, BC V6T 1Z2, Canada}
\email{sujatha@math.ubc.ca}

\begin{document}

\begin{abstract}
We study congruences between special values of zeta functions and modular
\(L\)-functions in cyclotomic towers.  The underlying mechanism is an
integral norm congruence for finite-level Iwasawa-theoretic elements,
which becomes stronger as one ascends the tower.  In the
\(\mathrm{GL}_1\) setting, this gives congruences for Dedekind zeta
values over totally real fields, with applications to higher
\(K\)-groups, generalized Bernoulli numbers, and Euler--Poincar\'e
characteristics.  In the \(\mathrm{GL}_2\) setting, we use
Mazur--Tate elements and their distribution relations to control the
primitive twisted special values appearing at each cyclotomic level.
Via Artin formalism, these give congruences for the corresponding
base-change \(L\)-values over successive cyclotomic fields.
\end{abstract}
\maketitle

\section{Introduction}

\par Zeta and $L$-functions are fundamental yet mysterious arithmetic objects, and congruences between their special values often provide deep arithmetic insights. The classical starting point is Kummer's
congruences for Bernoulli numbers. The study of the $p$-adic variation of generalized Bernoulli numbers led to the construction of the Kubota--Leopoldt \(p\)-adic
\(L\)-function \cite{KubotaLeopoldt} and to the systematic use of
\(p\)-adic measures and Iwasawa algebras, see for instance
\cite{Iwasawa1972,Washington1997,CoatesSujatha}.  For totally real
fields, the corresponding theory was developed by Cassou-Nogu\`es
\cite{CassouNogues}, Barsky \cite{Barsky}, and Deligne--Ribet
\cite{DeligneRibet}.
\par Let \(\Q_m/\mathbb Q\) be the subfield of degree
\(p^m\) in the cyclotomic \(\mathbb Z_p\)-extension of \(\mathbb Q\),
with \(p\geq5\). Clozel proved that
\[
        \zeta_{\Q_{m+1}}(-1)
        \equiv
        \zeta_{\Q_m}(-1)
        \pmod p;
\]
see \cite[Proposition~2.5]{ClozelCyclotomic}. The proof is geometric in
nature. Clozel considers finite Shimura sets attached to a definite quaternion
algebra along the cyclotomic tower. Counting Galois orbits gives a
congruence modulo \(p\) between the cardinalities at successive levels, while
a calculation of Tamagawa measures relates these cardinalities to the
special values \(\zeta_{\Q_m}(-1)\).  In the sequel
\cite{ClozelIwasawa}, Clozel explained that Coates had observed a different and more direct proof using Iwasawa's construction of $p$-adic $L$-functions. This viewpoint is closer to the norm argument developed here. We extend these congruences to higher powers of $p$, to a broader class of number fields, and to special zeta values at all negative odd integers. This constitutes the $\mathrm{GL}_1$ case of our method, which we then extend to a $\mathrm{GL}_2$ setting.

\par Let \(L\) be a totally real number field and let
\(\mathbb Q_\infty/\mathbb Q\) be the cyclotomic
\(\mathbb Z_p\)-extension. There is a unique integer \(e_L\geq0\)
such that $L\cap\mathbb Q_\infty
        =
        \mathbb Q_{e_L}$. Throughout, we use the relative
indexing $L^{(m)}
        :=
        L\mathbb Q_{e_L+m}$. Note that $\op{Gal}(L^{(m)}/L)\simeq \Z/p^m\Z$.
\begin{lthm}[Theorem~\ref{thm:totally-real-zeta-congruence}]
\label{lthm a}
Let \(L\) be a totally real number field, let \(p\) be an odd prime,
and write \(L^{(m)}=L\mathbb Q_{e_L+m}\), where
\(L\cap\mathbb Q_\infty=\mathbb Q_{e_L}\).  Let \(n\geq2\) be even
and suppose that \(d_{L,p}:=[L(\mu_p):L]\) does not divide \(n\). Then \(\zeta_{L^{(m)}}(1-n)\in\mathbb Z_p\) for every
\(m\geq0\), and \begin{equation}\label{thmacongruence}\zeta_{L^{(m+1)}}(1-n)\equiv\zeta_{L^{(m)}}(1-n)
\pmod{p^{m+1}}.\end{equation}
\end{lthm}
\noindent We emphasize that the congruence in fact holds for every number field
\(L\).  Indeed, if \(L\) is not totally real, then Lemma~\ref{lem:vanishing-nontotally-real} shows that
both sides of \eqref{thmacongruence} vanish identically. Thus the result vacuously holds when $L$ is not totally real. Our results may also be viewed in the context of base change for
\(p\)-adic zeta functions.  Related ideas appear in Iwasawa's
construction of \(p\)-adic \(L\)-functions and, more generally, in the
norm and transfer relations for \(p\)-adic zeta functions studied by
Ritter--Weiss and Kakde; see
\cite[\S3,6--7]{IwasawaPadicL},
\cite[\S2]{RitterWeissPseudomeasures}, and
\cite[Theorem~94 and Proposition~95]{KakdeTotallyReal}. \par The second part of the paper concerns special values of modular
\(L\)-functions in the cyclotomic tower.  Let \(f\) be a normalized
cuspidal newform of weight \(k>2\), trivial nebentype and level
\(\Gamma_0(N)\), let \(p\nmid N\) be an odd prime, and put $K_n=\Q(\mu_{p^{n+1}})$.
By Artin formalism, \(\mathrm{L}_{K_n}(f,s)\) factors as a product of Dirichlet
twists of \(f\), and the new characters appearing between
\(K_{n-1}\) and \(K_n\) are precisely the primitive characters of
conductor \(p^{n+1}\).  We study these new factors using Mazur--Tate
elements. The study of Mazur--Tate elements has gained considerable momentum, particularly concerning their Iwasawa invariants and their congruence properties; see \cite{MazurTate,PollackWeston}.

\par Let \(\mathcal O\) be the valuation ring of the completion of the
Hecke field of \(f\) at a prime above \(p\), with maximal ideal
\(\mathfrak m\) and residue field \(\kappa\). Let $\mathcal X_n:=\widehat{\Gal(K_n/\Q)}$ and for \(\varepsilon\in\{\pm1\}\), set
\[
        \mathcal X_n^\varepsilon
        :=
        \{\chi\in\mathcal X_n:\chi(-1)=\varepsilon\}.
\]
For \(\chi\in\mathcal X_n\), let \(\chi^\circ\) denote the associated
primitive Dirichlet character. Let \(\tau(\chi)\) denote the Gauss sum
of a primitive Dirichlet character \(\chi\), and set \(\tau(\mathbf1)=1\). Using the
cohomological periods \(\Omega_f^\pm\), define
\[
        \mathrm{L}_{K_n}^{\mathrm{alg}}(f,1)
        :=
        \left(
        \prod_{\chi\in\mathcal X_n}\tau(\chi^\circ)
        \right)
        \frac{\mathrm{L}_{K_n}(f,1)}
        {(\Omega_f^+)^{\#\mathcal X_n^+}
         (\Omega_f^-)^{\#\mathcal X_n^-}}.
\]
With this normalization, these algebraic values are \(p\)-integral.

\par Our main result shows that, in the ordinary non-anomalous case, the
algebraic special values are constant modulo the maximal ideal.

\begin{lthm}
[Theorem \ref{lthm:intro-modular}]\label{thm b}
Assume that \(f\) is ordinary at \(p\),
\(\kappa=\F_p\), and
\(a_p(f)\not\equiv1\pmod{\mathfrak m}\).  Then
\(\mathrm{L}_{K_n}^{\mathrm{alg}}(f,1)\in\mathcal O\) for every \(n\geq0\), and
\[
        \mathrm{L}_{K_n}^{\mathrm{alg}}(f,1)
        \equiv
        \mathrm{L}_{K_{n-1}}^{\mathrm{alg}}(f,1)\pmod{\mathfrak m}
\]
for every \(n\geq1\). If \(\mathcal O=\Z_p\), these are congruences modulo \(p\).
\end{lthm}

\par We record three applications of Theorem~\ref{lthm a}.  The first two
concern totally real abelian fields, while the third applies to arbitrary
totally real fields.

\par The first application concerns higher even \(K\)-groups.  For totally
real abelian fields, the Birch--Tate--Lichtenbaum formula relates
\(\#K_{2n-2}(\mathcal O_F)\), for even \(n\geq2\), to
\(\zeta_F(1-n)\); see
\cite{Garland1971,Quillen1973,Borel1977,Lichtenbaum1972,Lichtenbaum1973}. Theorem~\ref{thm:K-group-congruence-application} implies
that, for \(1\leq r\leq m+1\),
\[
        p^r\mid\#K_{2n-2}(\mathcal O_{L^{(m+1)}})
        \quad\text{if and only if}\quad
        p^r\mid\#K_{2n-2}(\mathcal O_{L^{(m)}}).
\]
In particular, the \(p\)-primary part is trivial at one level if and
only if it is trivial at every level.  Conversely, if
\(K_{2n-2}(\mathcal O_L)[p^\infty]\neq0\), then
Corollary~\ref{cor:K-growth} shows that
\[
        p^{m+1}\mid
        \#K_{2n-2}(\mathcal O_{L^{(m)}})
\]
for every \(m\geq0\).

\par For totally real abelian $L/\Q$, Theorem~\ref{thm:Bernoulli-unit-criterion} gives a level-independent criterion for when the generalized Bernoulli numbers (see \cite[Chapter~4]{Washington1997}) are $p$-adic units.

\par The third application concerns Euler--Poincar\'e characteristics. Let \(\chi_{\mathrm{EP}}(H)\) denote the Euler--Poincar\'e
characteristic of an arithmetic group \(H\). For
background on arithmetic groups and their basic properties, see
\cite{BorelArithmeticGroups,PlatonovRapinchuk}; for their
Euler--Poincar\'e characteristics, see \cite{Harder1971}. Harder's Gauss--Bonnet formula \cite{Harder1971} gives, for every
totally real field \(F\),
\[
        \chi_{\mathrm{EP}}
        \bigl(\operatorname{Sp}_{2g}(\mathcal O_F)\bigr)
        =
        \prod_{j=1}^{g}\zeta_F(1-2j).
\]If
\(d_{L,p}\nmid 2j\) for \(1\leq j\leq g\), then
Theorem~\ref{thm:Euler-characteristic-congruence} yields
\[
        \chi_{\mathrm{EP}}
        \bigl(\operatorname{Sp}_{2g}(\mathcal O_{L^{(m+1)}})\bigr)
        \equiv
        \chi_{\mathrm{EP}}
        \bigl(\operatorname{Sp}_{2g}(\mathcal O_{L^{(m)}})\bigr)
        \pmod{p^{m+1}}.
\]
Thus the same congruence principle has consequences for higher
\(K\)-groups, generalized Bernoulli numbers, and arithmetic
Euler--Poincar\'e characteristics.

\par Several natural questions remain. In the $\op{GL}_1$-case, it would be
of interest to understand the exceptional Teichm\"uller branches excluded
from Theorem~\ref{thm:totally-real-zeta-congruence}. The obstruction in
these cases is that the relevant interpolation does not admit the same
integral power-series description, so the norm argument used here no
longer applies directly. A second direction is to retain the
equivariant information throughout the construction, rather than passing
immediately to scalar character evaluations and their products.  In a noncommutative setting one expects the scalar
characters to be replaced by finite-dimensional Artin representations
and the relevant \(p\)-adic \(L\)-function to take values in a localized
\(K_1\)-group; see, for example, \cite{SujathaReductions}.
\par The mechanism underlying our arguments should apply much more
generally whenever suitably normalized special values are interpolated by
an integral $p$-adic $L$-function.  In particular, the conjectural
framework of Coates--Perrin-Riou predicts the existence of such
$p$-adic $L$-functions for a broad class of ordinary motives admitting
critical values, with interpolation formulae involving the corresponding
complex $L$-values, periods, and local Euler factors; see
\cite{CoatesPerrinRiou,CoatesMotivesII}, and also the Bourbaki
exposition \cite{CoatesBourbaki}.  Whenever the resulting interpolation
object is represented by an integral Iwasawa measure, the cyclotomic
norm argument developed here suggests corresponding congruences between
special values along cyclotomic towers.  Thus the $\mathrm{GL}_1$ and
$\mathrm{GL}_2$ results of the present paper may be viewed as instances
of a more general phenomenon predicted by the theory of motivic
$p$-adic $L$-functions.

\par The paper is organized as follows.  Section~\ref{s 2} develops the elementary
cyclotomic norm congruence which provides the algebraic mechanism behind
the later arguments.  In Section~\ref{s 3} we recall the
Deligne--Ribet \(p\)-adic zeta function and prove
Theorem~\ref{lthm a}, giving congruences between special zeta values in
cyclotomic towers.  Section~\ref{s 4} records several arithmetic applications,
including consequences for higher \(K\)-groups, generalized Bernoulli
numbers, and Euler--Poincar\'e characteristics.  Finally
in Section~\ref{sec:Mazur-Tate}, we develop the analogous \(\mathrm{GL}_2\)
construction using Mazur--Tate elements and prove Theorem~\ref{thm b}.
\section{Preliminaries}\label{s 2}
\subsection{Zeta values and the functional equation}
\par Let \(K\) be a number field of discriminant \(D_K\) and signature $\bigl(r_1(K),r_2(K)\bigr)$,
where \(r_1(K)\) denotes the number of real embeddings of \(K\) and
\(r_2(K)\) denotes the number of pairs of complex-conjugate embeddings.
Thus
\[
        [K:\mathbb Q]=r_1(K)+2r_2(K).
\]
For \(\operatorname{Re}(s)>1\), the Dedekind zeta function of \(K\) is
defined by
\[
        \zeta_K(s)
        =
        \sum_{\mathfrak a\subseteq\mathcal O_K}
        \frac{1}{\operatorname{N}\mathfrak a^{\,s}}
        =
        \prod_{\mathfrak q}
        \left(
        1-\operatorname{N}\mathfrak q^{-s}
        \right)^{-1},
\]
where the sum runs over the nonzero integral ideals of
\(\mathcal O_K\), and the product runs over the nonzero prime ideals of
\(\mathcal O_K\). Introduce the archimedean gamma factors
\[
        \Gamma_{\mathbb R}(s)
        :=
        \pi^{-s/2}\Gamma\left(\frac{s}{2}\right)\quad
\text{and}\quad
        \Gamma_{\mathbb C}(s)
        :=
        2(2\pi)^{-s}\Gamma(s).
\]
The completed Dedekind zeta function is defined as follows:
\[
        \Lambda_K(s)
        :=
        |D_K|^{s/2}
        \Gamma_{\mathbb R}(s)^{r_1(K)}
        \Gamma_{\mathbb C}(s)^{r_2(K)}
        \zeta_K(s).
\]

\par The function \(\zeta_K(s)\) admits a meromorphic continuation to
the whole complex plane, with a unique simple pole at \(s=1\). The
completed zeta function satisfies the functional equation
\[
        \Lambda_K(s)=\Lambda_K(1-s).
\]
It has simple poles at \(s=0,1\) and $s(s-1)\Lambda_K(s)$ is entire.

\par The poles of the archimedean gamma factors account for the
trivial zeros of \(\zeta_K(s)\). Recall that \(\Gamma(s)\) has a
simple pole at every nonpositive integer and has no zeros. Therefore for a positive integer $j$,
\begin{equation}\label{zeta(K)order}\operatorname{ord}_{s=-j}\zeta_K(s)
        =
        \begin{cases}
        r_2(K), & j \text{ odd},\\[4pt]
        r_1(K)+r_2(K), & j \text{ even}.
        \end{cases}
\end{equation}
These are called the \emph{trivial zeros} of the Dedekind zeta
function.

\begin{lemma}
\label{lem:vanishing-nontotally-real}
Let \(K\) be a number field and let $n$ be a positive even integer. Then
\[
        \operatorname{ord}_{s=1-n}\zeta_K(s)=r_2(K).
\]
In particular, if $K$ is not totally real, then 
\[\zeta_K(1-n)=0.\]
\end{lemma}

\begin{proof} The result is a special case of \eqref{zeta(K)order}.
\end{proof}

\subsection{Cyclotomic norms of integral power series}\label{sec:norm}

\par The arithmetic of cyclotomic extensions is closely related to the behaviour of
norms of local units. We briefly recall some aspects of the theory which will be used throughout the paper. For general background on cyclotomic
fields, local units, and their norm maps, we refer to
\cite[Chapters~1--2]{CoatesSujatha}; see also
\cite[Chapter~13]{Washington1997}. The passage from norm-compatible systems
to elements of the Iwasawa algebra, and equivalently to \(p\)-adic measures,
is discussed in \cite[Chapter~3]{CoatesSujatha}.
\par Let $p$ be an odd prime number and $v_p$ be the $p$-adic valuation normalized by $v_p(p)=1$. Denote by $R_{p^m}$ the set of all $p^m$-th roots of unity. Denote by
$P_{p^{m+1}}$ the set of primitive $p^{m+1}$-st roots of unity. For
$G(T)\in\mathbb Z_p\llbracket T\rrbracket$, set
\begin{equation}\label{A_m and H_m}
 A_m(G)=\prod_{\zeta\in R_{p^m}}G(\zeta-1)
 \qquad\text{and}\qquad
 H_m(G)=\prod_{\zeta\in P_{p^{m+1}}}G(\zeta-1).
\end{equation}
The Galois group $\op{Gal}\bigl(\Q_p(\mu_{p^{m+1}})/\Q_p\bigr)$
permutes the factors in both products. Hence \(A_m(G)\) and \(H_m(G)\) are Galois invariant and lie in \(\Q_p\). The elementary decomposition
$R_{p^{m+1}}=R_{p^m}\sqcup P_{p^{m+1}}$ gives
\begin{equation}\label{eq:AmHm}
 A_{m+1}(G)=A_m(G)H_m(G).
\end{equation}

\begin{lemma}\label{lem:cyclotomic-norm-unit}
For $G(T)\in\mathbb Z_p\llbracket T\rrbracket$, then the following assertions hold:
\begin{enumerate}
\item if $G(0)\in\mathbb Z_p^\times$, then
\[
 H_m(G)\equiv1\pmod{p^{m+1}\mathbb Z_p};
\]
\item if $G(0)\in p\mathbb Z_p$, then $H_m(G)\in p\mathbb Z_p$ and 
\[v_p(A_m(G))\geq m+1.
\]
\end{enumerate}
\end{lemma}

\begin{proof}
Suppose first that $G(0)$ is a unit and fix
$\zeta\in P_{p^{m+1}}$.  Then $G(\zeta-1)$ is a unit in
$\mathbb Z_p[\zeta]$.  The product of its Galois conjugates is precisely
$H_m(G)$.  The extension
$\mathbb Q_p(\mu_{p^{m+1}})/\mathbb Q_p$ is the local cyclotomic
extension of conductor $p^{m+1}$, and local Class field theory identifies the norm
of its unit group with the subgroup whose principal-unit component lies in
$1+p^{m+1}\mathbb Z_p$ from which we obtain the asserted congruence.

Now suppose that $G(0)\in p\mathbb Z_p$.  For every $p$-power root of
unity $\zeta$ we have
$G(\zeta-1)\equiv G(0)\pmod{\zeta-1}$, so each such value is a nonunit.
Consequently $H_m(G)$ has positive $p$-adic valuation.  Iterating
\eqref{eq:AmHm} gives
\[
 A_m(G)=G(0)\prod_{j=0}^{m-1}H_j(G),
\]
and every factor on the right is divisible by $p$.
\end{proof}

\begin{theorem}\label{thm:abstract-norm}
For every $G(T)\in\mathbb Z_p\llbracket T\rrbracket$ and every $m\geq0$,
\[
 A_{m+1}(G)\equiv A_m(G)
 \pmod{p^{m+1}\mathbb Z_p}.
\]
\end{theorem}

\begin{proof}
If $G(0)$ is a unit, combine \eqref{eq:AmHm} with the first part of
Lemma~\ref{lem:cyclotomic-norm-unit}.  If $G(0)$ is a nonunit, the second
part gives $v_p(A_m(G))\geq m+1$ and $H_m(G)\in p\mathbb Z_p$; hence
\[
 A_{m+1}(G)-A_m(G)=A_m(G)(H_m(G)-1)
\]
has valuation at least $m+1$.
\end{proof}

The coefficient ring of a modular $p$-adic $L$-function is usually larger
than $\mathbb Z_p$.  A direct assertion that $H_m(G)\equiv1$ modulo
$p^{m+1}$ is then false in general: the reduction of the constant term can
lie in a residue field larger than $\mathbb F_p$.  The correct scalar
statement is obtained by taking the coefficient-field norm.

\begin{definition}\label{def:coefficient-field-norm}
Let $E/\mathbb Q_p$ be a finite extension with ring of integers
$\mathcal O_E$, and let
\[
        G(T)\in\mathcal O_E\llbracket T\rrbracket.
\]
We define the \emph{coefficient-field norm} of $G$ by
\[
        \mathcal N_{E/\mathbb Q_p}G(T)
        :=
        \prod_{\sigma:E\hookrightarrow\overline{\mathbb Q}_p}
        \sigma(G(T)),
\]
where $\sigma(G(T))$ denotes the power series obtained by applying
$\sigma$ to every coefficient of $G(T)$.  Since the product is invariant
under $\operatorname{Gal}(\overline{\mathbb Q}_p/\mathbb Q_p)$ and has
integral coefficients, one has
\[
        \mathcal N_{E/\mathbb Q_p}G(T)
        \in\mathbb Z_p\llbracket T\rrbracket.
\]

For $m\geq0$, we also put
\[
        A_{m,E}(G)
        :=
        \prod_{\zeta^{p^m}=1}G(\zeta-1).
\]
This product belongs to $\mathcal O_E$: indeed, it is a symmetric
expression in the $p^m$-th roots of unity with coefficients in
$\mathcal O_E$.
\end{definition}

\begin{proposition}\label{prop:coefficient-norm}
Let $E/\mathbb Q_p$ be a finite extension, let
$G(T)\in\mathcal O_E\llbracket T\rrbracket$, and let
$\mathcal N_{E/\mathbb Q_p}G$ and $A_{m,E}(G)$ be as in
Definition~\ref{def:coefficient-field-norm}.  Then, for every $m\geq0$,
\begin{equation}\label{eq:coefficient-norm-Am}
        N_{E/\mathbb Q_p}\bigl(A_{m,E}(G)\bigr)
        =
        A_m\bigl(\mathcal N_{E/\mathbb Q_p}G\bigr).
\end{equation}
Consequently,
\begin{equation}\label{eq:coefficient-norm-congruence}
\begin{split}
        N_{E/\mathbb Q_p}\bigl(A_{m+1,E}(G)\bigr)
        &\equiv
        N_{E/\mathbb Q_p}\bigl(A_{m,E}(G)\bigr)\\
        &\pmod{p^{m+1}\mathbb Z_p}.
\end{split}
\end{equation}
\end{proposition}

\begin{proof}
For every $\mathbb Q_p$-embedding
$\sigma:E\hookrightarrow\overline{\mathbb Q}_p$, choose an extension of
$\sigma$ to $\overline{\mathbb Q}_p$.  Since such an extension permutes
the set of $p^m$-th roots of unity, we have
\[
\begin{aligned}
        \sigma\bigl(A_{m,E}(G)\bigr)
        &=
        \sigma\left(
        \prod_{\zeta^{p^m}=1}G(\zeta-1)
        \right)\\
        &=
        \prod_{\zeta^{p^m}=1}
        \sigma(G)(\zeta-1).
\end{aligned}
\]
Taking the product over all embeddings $\sigma$ and interchanging the
two finite products gives
\[
\begin{aligned}
        N_{E/\mathbb Q_p}\bigl(A_{m,E}(G)\bigr)
        &=
        \prod_{\sigma:E\hookrightarrow\overline{\mathbb Q}_p}
        \prod_{\zeta^{p^m}=1}
        \sigma(G)(\zeta-1)\\
        &=
        \prod_{\zeta^{p^m}=1}
        \prod_{\sigma:E\hookrightarrow\overline{\mathbb Q}_p}
        \sigma(G)(\zeta-1)\\
        &=
        \prod_{\zeta^{p^m}=1}
        \bigl(\mathcal N_{E/\mathbb Q_p}G\bigr)(\zeta-1)\\
        &=
        A_m\bigl(\mathcal N_{E/\mathbb Q_p}G\bigr).
\end{aligned}
\]
This proves \eqref{eq:coefficient-norm-Am}.  Since
$\mathcal N_{E/\mathbb Q_p}G(T)\in\mathbb Z_p\llbracket T\rrbracket$,
Theorem~\ref{thm:abstract-norm} gives
\[
        A_{m+1}\bigl(\mathcal N_{E/\mathbb Q_p}G\bigr)
        \equiv
        A_m\bigl(\mathcal N_{E/\mathbb Q_p}G\bigr)
        \pmod{p^{m+1}\mathbb Z_p}.
\]
Using \eqref{eq:coefficient-norm-Am} at levels $m+1$ and $m$ yields
\eqref{eq:coefficient-norm-congruence}.
\end{proof}

\par Let $\mathbb Q_\infty/\mathbb Q$ denote the cyclotomic \(\mathbb Z_p\)-extension of $\Q$, and for \(m\geq0\) let
\(\mathbb Q_m\) be its unique subextension of degree \(p^m\) over
\(\mathbb Q\). We set $\Gamma
        :=
        \operatorname{Gal}(\mathbb Q_\infty/\mathbb Q)$ and fix a topological generator \(\gamma\) of \(\Gamma\).  The Iwasawa algebra
of \(\Gamma\) is the completed group ring
\[
        \Lambda(\Gamma)
        :=
        \mathbb Z_p\llbracket\Gamma\rrbracket
        =
        \varprojlim_m
        \mathbb Z_p[\Gamma/\Gamma^{p^m}],
\]
where the transition maps are induced by the natural quotient maps
\[
        \Gamma/\Gamma^{p^{m+1}}
        \longrightarrow
        \Gamma/\Gamma^{p^m}.
\]
We have an isomorphism
\begin{equation}\label{eq:iwasawa-power-series}
\iota_\gamma:\Lambda(\Gamma)
\xrightarrow{\;\sim\;}
\mathbb Z_p\llbracket T\rrbracket,
\end{equation}
defined by $\gamma\longmapsto 1+T$. For each \(m\geq 0\), set
\[
\nu_m(T)=(1+T)^{p^m}-1.
\]
Under the identification \eqref{eq:iwasawa-power-series}, the natural
projection
\[
\Lambda(\Gamma)\longrightarrow
\mathbb Z_p[\Gamma/\Gamma^{p^m}]
\]
corresponds to reduction modulo \(\nu_m(T)\).  Thus there is a canonical
isomorphism
\[
\mathbb Z_p[\Gamma/\Gamma^{p^m}]
\simeq
\mathbb Z_p\llbracket T\rrbracket/(\nu_m(T)),
\]
and, for \(G\in\mathbb Z_p\llbracket T\rrbracket\), we denote by
\(G_m\) its image modulo \(\nu_m(T)\).  The characters of the cyclic group
\(\Gamma/\Gamma^{p^m}\) are parametrized by the \(p^m\)-th roots of
unity, i.e., a character \(\chi\) is determined by $\chi(\gamma)=\zeta$, where $\zeta^{p^m}=1$.
Evaluation at
\(\chi\) is given by
\[
        \chi(G_m)
        =
        G\bigl(\chi(\gamma)-1\bigr)
        =
        G(\zeta-1).
\]
Consequently,
\[
        A_m(G)
        =
        \prod_{\zeta^{p^m}=1}G(\zeta-1)
        =
        \prod_{\chi\in
        \widehat{\Gamma/\Gamma^{p^m}}}\chi(G_m).
\]
Let $\Delta_m:=\Gamma/\Gamma^{p^m}$ and regard \(G_m\) as an element of \(\mathbb Z_p[\Delta_m]\).
Multiplication by \(G_m\) defines a \(\mathbb Z_p\)-linear endomorphism
\[
        m_{G_m}:
        \mathbb Z_p[\Delta_m]
        \longrightarrow
        \mathbb Z_p[\Delta_m],
        \qquad
        x\longmapsto G_mx.
\]
After extending scalars to a field containing the values of all characters
of \(\Delta_m\), the regular representation decomposes as the direct sum of
its one-dimensional characters.  Hence
\[
\begin{aligned}
        \det(m_{G_m})
        &=
        \prod_{\chi\in\widehat{\Delta_m}}\chi(G_m)\\
        &=
        \prod_{\zeta^{p^m}=1}G(\zeta-1)
        =
        A_m(G).
\end{aligned}
\]
Thus \(A_m(G)\) is precisely the regular determinant of the finite-level
specialization of the Iwasawa-algebra element \(G\).  In particular, the
passage from \(A_m(G)\) to \(A_{m+1}(G)\) amounts to comparing regular
determinants at two successive quotients of the cyclotomic Iwasawa
algebra.  The additional factor is contributed by the characters which
are new at level \(m+1\), namely those of exact order \(p^{m+1}\).
This determinant interpretation will reappear in the final section, where complete character products of Mazur--Tate components are again identified with regular determinants.
\section{Iwasawa theory and interpolation of $L$-values}\label{s 3}

\par We now recall the form of cyclotomic Iwasawa theory that will be used to
interpolate the special values appearing below.  The basic point is that
the Iwasawa algebra of the cyclotomic Galois group may be identified with
the algebra of bounded $p$-adic measures, and, after choosing a
topological generator, with a one-variable power-series ring.  Under these
identifications, evaluation at finite-order characters becomes evaluation
of the corresponding Iwasawa power series at $p$-power roots of unity.
This is the bridge between the $p$-adic $L$-functions considered in this
section and the cyclotomic norm congruences proved in
Section~\ref{sec:norm}.
\par For the classical construction over $\mathbb Q$, see \cite{KubotaLeopoldt,Iwasawa1972,Washington1997}; for totally real fields, see \cite{CassouNogues,Barsky,DeligneRibet,RibetTotallyReal}.  The relation between Iwasawa theory and special values over totally real fields is also discussed systematically in \cite{CoatesSujatha}, while the main conjectural picture is proved in great generality by Wiles \cite{Wiles1990}.  We recall only the parts of this theory which enter the norm argument.
\par Let $L$ be a totally real number field and \(p\) be an odd prime number. We set $L_\infty:=L\mathbb Q_\infty$. This is the cyclotomic \(\mathbb Z_p\)-extension of \(L\). There is a unique integer
\(e_L\geq0\) such that
\begin{equation}
\label{eq:totally-real-eL}
        L\cap\mathbb Q_\infty
        =
        \mathbb Q_{e_L}.
\end{equation}
We define $L^{(m)}
        :=
        L\mathbb Q_{e_L+m}$ for $m\geq0$.
Then $L^{(0)}=L$ and
\[
        \operatorname{Gal}(L^{(m)}/L)
        \simeq
        \operatorname{Gal}
        (\mathbb Q_{e_L+m}/\mathbb Q_{e_L})
        \simeq
        \mathbb Z/p^m\mathbb Z.
\]
\noindent The resulting tower is depicted as follows.

\begin{figure}[ht]
\centering
\begin{tikzpicture}[x=2.5cm,y=1.5cm]
\node (Q0) at (0,0) {$\mathbb Q_{e_L}$};
\node (Q1) at (0,1) {$\mathbb Q_{e_L+1}$};
\node (Qm) at (0,2) {$\mathbb Q_{e_L+m}$};
\node (Qd) at (0,3) {$\vdots$};
\node (Qi) at (0,4) {$\mathbb Q_\infty$};

\node (L0) at (1.8,0) {$L^{(0)}$};
\node (L1) at (1.8,1) {$L^{(1)}$};
\node (Lm) at (1.8,2) {$L^{(m)}$};
\node (Ld) at (1.8,3) {$\vdots$};
\node (Li) at (1.8,4) {$L_\infty=L\mathbb Q_\infty$};

\draw (0,0.2) -- (0,0.8);
\draw (0,1.2) -- (0,1.8);
\draw (0,2.2) -- (0,2.8);
\draw (0,3.2) -- (0,3.8);

\draw (1.8,0.2) -- (1.8,0.8);
\draw (1.8,1.2) -- (1.8,1.8);
\draw (1.8,2.2) -- (1.8,2.8);
\draw (1.8,3.2) -- (1.8,3.8);

\draw (0.45,0) -- (1.35,0);
\draw (0.45,1) -- (1.35,1);
\draw (0.45,2) -- (1.35,2);
\draw (0.45,4) -- (1.35,4);
\end{tikzpicture}
\caption{The cyclotomic tower relative to the totally real field \(L\).
The integer \(e_L\) is characterized by
\(L\cap\mathbb Q_\infty=\mathbb Q_{e_L}\), and
\(L^{(m)}=L\mathbb Q_{e_L+m}\).}
\label{fig:cyclotomic-tower}
\end{figure}
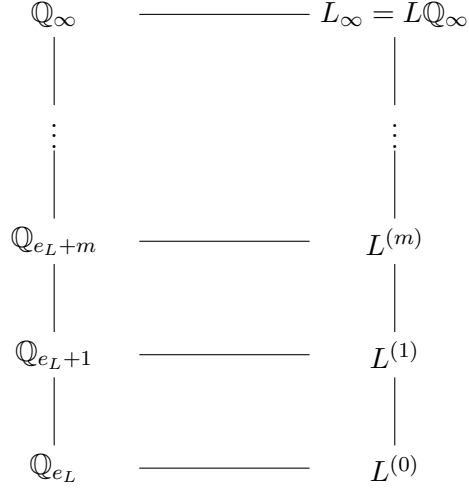

\par Setting $\Gamma_L
        :=
        \operatorname{Gal}(L_\infty/L)$, we have that $\Gamma_L\simeq\mathbb Z_p$. Choose a topological generator $\gamma_L\in\Gamma_L$. For \(m\geq0\), set $\Gamma_{L,m}
        :=
        \Gamma_L/\Gamma_L^{p^m}$. The natural restriction map identifies $\Gamma_{L,m}
        \simeq
        \operatorname{Gal}(L^{(m)}/L)$. We denote its character group by
\[
        \Psi_{L,m}
        :=
        \operatorname{Hom}
        \bigl(
        \Gamma_{L,m},
        \overline{\mathbb Q}_p^{\,\times}
        \bigr).
\]
Since \(\Gamma_{L,m}\) is cyclic of order \(p^m\), evaluation at
\(\gamma_L\) gives a bijection
\[
        \Psi_{L,m}
        \xrightarrow{\sim}
        \mu_{p^m}
\]
defined by $\chi\mapsto\chi(\gamma_L)$.

\par Let \(\mathcal O\) be the valuation ring of a finite extension of
\(\mathbb Q_p\).  The Iwasawa algebra of \(\Gamma_L\) is
\[
        \Lambda_{\mathcal O}(\Gamma_L)
        :=
        \mathcal O\llbracket\Gamma_L\rrbracket
        =
        \varprojlim_m
        \mathcal O[\Gamma_{L,m}].
\]
As in \eqref{eq:iwasawa-power-series}, the choice of \(\gamma_L\) gives an isomorphism
\begin{equation}
\label{eq:totally-real-iwasawa-power-series}
        \iota_{\gamma_L}:
        \mathcal O\llbracket\Gamma_L\rrbracket
        \xrightarrow{\;\sim\;}
        \mathcal O\llbracket T\rrbracket
\end{equation}
defined by $\gamma_L\mapsto1+T$. The Iwasawa algebra may equally well be regarded as an algebra
of bounded measures on \(\Gamma_L\).  If
\[
        \lambda=(\lambda_m)_m
        \in
        \mathcal O\llbracket\Gamma_L\rrbracket
\]
and
\[
        \lambda_m
        =
        \sum_{a\in\Gamma_{L,m}}
        c_{a,m}[a],
\]
then the compatibility under the transition maps defines a bounded
\(\mathcal O\)-valued measure \(\mu_\lambda\) by
\[
        \mu_\lambda
        \bigl(a\Gamma_L^{p^m}\bigr)
        =
        c_{a,m}.
\]
Conversely, every bounded measure determines an element of the Iwasawa
algebra.  Under the identification $\Gamma_L\simeq\mathbb Z_p$ defined by $x\mapsto\gamma_L^x$, the associated power series is the Amice transform
\[
        \mathcal A_\mu(T)
        =
        \int_{\mathbb Z_p}(1+T)^x\,d\mu(x).
\]
Hence, if \(\chi\in\Psi_{L,m}\), then
\[
        \int_{\Gamma_L}\chi\,d\mu
        =
        \mathcal A_\mu
        \bigl(\chi(\gamma_L)-1\bigr).
\]
This is the measure-theoretic origin of the evaluations at roots of
unity which occur in the norm argument.

\par We give a measure-theoretic interpretation of the identity
\eqref{eq:AmHm}.  Let \(\mu\) be a bounded measure on \(\Gamma_L\), and
let
\[
        \pi_m:\Gamma_L\longrightarrow
        \Gamma_{L,m}=\Gamma_L/\Gamma_L^{p^m}
\]
be the natural quotient map.  Its pushforward
\(\mu_m:=(\pi_m)_*\mu\) may be regarded as the group-ring element
\[
        \mu_m
        =
        \sum_{a\in\Gamma_{L,m}}
        \mu\bigl(a\Gamma_L^{p^m}\bigr)[a]
        \in\mathcal O[\Gamma_{L,m}].
\]
For \(\chi\in\Psi_{L,m}\), evaluation at \(\chi\) gives
\[
        \chi(\mu_m)
        =
        \sum_{a\in\Gamma_{L,m}}
        \mu\bigl(a\Gamma_L^{p^m}\bigr)\chi(a)
        =
        \int_{\Gamma_L}\chi\,d\mu.
\]
Now identify \(\Gamma_L\) with \(\Z_p\) by
\(x\mapsto\gamma_L^x\).  If
\(\chi(\gamma_L)=\zeta\), then
\[
        \chi(\gamma_L^x)
        =
        \zeta^x
        =
        \bigl(1+(\zeta-1)\bigr)^x.
\]
Hence, from the definition
\[
        \mathcal A_\mu(T)
        =
        \int_{\Z_p}(1+T)^x\,d\mu(x),
\]
we obtain
\[
        \chi(\mu_m)
        =
        \int_{\Gamma_L}\chi\,d\mu
        =
        \mathcal A_\mu(\zeta-1)
        =
        \mathcal A_\mu\bigl(\chi(\gamma_L)-1\bigr).
\]
Thus evaluation of the finite-level group-ring element at a character
is the same as evaluating the Amice transform at the corresponding
\(p\)-power root of unity minus \(1\).
\par On passing from level \(m\) to level \(m+1\), the character group
admits the disjoint decomposition
\[
        \Psi_{L,m+1}
        =
        \operatorname{Inf}_{\Gamma_{L,m}}^{\Gamma_{L,m+1}}
        (\Psi_{L,m})
        \sqcup
        \Psi_{L,m+1}^{\mathrm{new}},
\]
where
\[
        \Psi_{L,m+1}^{\mathrm{new}}
        :=
        \left\{
        \chi\in\Psi_{L,m+1}:
        \operatorname{ord}(\chi)=p^{m+1}
        \right\}
\]
and \(\operatorname{Inf}_{\Gamma_{L,m}}^{\Gamma_{L,m+1}}(\Psi_{L,m})\)
denotes the set of characters of \(\Gamma_{L,m+1}\) obtained by
composing characters in \(\Psi_{L,m}\) with the natural projection
\(\Gamma_{L,m+1}\twoheadrightarrow\Gamma_{L,m}\).
For a character inflated from \(\Psi_{L,m}\), compatibility of the
pushforward measures gives
\[
        \widehat\mu_{m+1}(\chi)
        =
        \widehat\mu_m(\chi).
\]
Thus the Fourier data at level \(m+1\) consist of the old character
evaluations already present at level \(m\), together with the genuinely
new evaluations at characters of exact order \(p^{m+1}\). If \(G(T)=\mathcal A_\mu(T)\), taking the product of all character
evaluations gives the identity \eqref{eq:AmHm}.

\par Let $\mathcal K_\infty
        :=
        L(\mu_{p^\infty})$ and set $\mathcal G_L
        :=
        \operatorname{Gal}(\mathcal K_\infty/L)$. Since \(L\) is totally real, complex conjugation belongs to
\(\mathcal G_L\).  The cyclotomic character gives an embedding of
\(\mathcal G_L\) into \(\mathbb Z_p^\times\), and, since \(p\) is odd,
there is a decomposition
\[
        \mathcal G_L
        \simeq
        \Delta_L\times\Gamma_L,
\]
where \(\Delta_L\) is the finite subgroup of order prime to \(p\) and
\(\Gamma_L\) is the pro-\(p\) part defined above.  We write $\kappa_L:
        \mathcal G_L
        \longrightarrow
        \mathbb Z_p^\times$ for the cyclotomic character and 
        \begin{equation}\label{omegaL}\omega_L:
        \Delta_L
        \longrightarrow
        \mu_{p-1}\end{equation} for its Teichm\"uller part.  We extend \(\omega_L\) to
\(\mathcal G_L\) by making it trivial on \(\Gamma_L\).
\par Let \(S_p\) denote the set of primes of \(L\) lying above \(p\).  For a
finite-order character $\chi:\mathcal G_L
        \longrightarrow
        \overline{\mathbb Q}_p^{\,\times}$,
let $\mathrm{L}_{L,S_p}(s,\chi)$ denote the corresponding Hecke, \(L\)-function
with the Euler factors at the primes in \(S_p\) removed.

\par We now recall the construction of Deligne--Ribet. We refer especially to
\cite[\S0, Theorems~(0.4) and~(0.5), pp.~231--232]{DeligneRibet}
for the measure-theoretic formulation, and to
\cite[\S8, Main Theorem~(8.2) and Theorem~(8.4),
pp.~278--283]{DeligneRibet}
for the general theorem and its proof by means of Hilbert modular
Eisenstein series. The spaces on which the relevant distributions and
measures are defined are introduced in pp.~238--245 of \emph{loc. cit.}
\par Let $\Lambda_L
        :=
        \mathbb Z_p\llbracket\mathcal G_L\rrbracket$. Since \(\mathcal G_L\) is profinite, an element of \(\Lambda_L\) may be
viewed either as a compatible family of elements in the finite group
rings $\mathbb Z_p[\mathcal G_L/U]$, where \(U\) ranges over the open normal subgroups of \(\mathcal G_L\),
or equivalently as a bounded \(\mathbb Z_p\)-valued measure on
\(\mathcal G_L\). Let \(S_p\) be the set of primes of \(L\) above
\(p\), and let \(M_{L,S_p}\) denote the maximal abelian extension of
\(L\) which is unramified at every finite prime outside \(S_p\). Setting $\mathfrak G_{L,S_p}
        :=
        \operatorname{Gal}(M_{L,S_p}/L)$, the cyclotomic extension
\[
        \mathcal K_\infty=L(\mu_{p^\infty})
\]
is contained in \(M_{L,S_p}\), and restriction therefore gives a
canonical surjection
\[
        \pi_{\mathrm{cyc}}:
        \mathfrak G_{L,S_p}
        \longrightarrow
        \mathcal G_L
        =
        \operatorname{Gal}(\mathcal K_\infty/L).
\]

\par A locally constant function in the Deligne--Ribet construction
means a locally constant function
\[
        \varepsilon:
        \mathfrak G_{L,S_p}
        \longrightarrow V,
\]
where \(V\) is a finite extension of \(\mathbb Q_p\), or more
generally a finite-dimensional \(\mathbb Q_p\)-vector space.  Since
\(\mathfrak G_{L,S_p}\) is profinite, this is equivalent to saying
that there is an open subgroup \(U\subseteq\mathfrak G_{L,S_p}\) and
a function
\[
        \overline{\varepsilon}:
        \mathfrak G_{L,S_p}/U
        \longrightarrow V
\]
such that
\[
        \varepsilon
        =
        \overline{\varepsilon}\circ
        \bigl(
        \mathfrak G_{L,S_p}
        \longrightarrow
        \mathfrak G_{L,S_p}/U
        \bigr).
\]
Thus \(\varepsilon\) factors through a finite strict ray class group.

\par Let \begin{equation}\label{kappadefn}\kappa:
        \mathfrak G_{L,S_p}
        \longrightarrow
        \mathbb Z_p^\times
        \end{equation}
denote the \(p\)-adic cyclotomic character.  For
\(c\in\mathfrak G_{L,S_p}\) and a locally constant function
\(\varepsilon:\mathfrak G_{L,S_p}\to V\), define $\varepsilon_c(g)
        :=
        \varepsilon(cg)$ where $g\in\mathfrak G_{L,S_p}$. Deligne and Ribet consider the modified special value
\[
        \Delta_c(1-k,\varepsilon)
        :=
        \mathrm{L}_{L,S_p}(1-k,\varepsilon)
        -
        \kappa(c)^k
        \mathrm{L}_{L,S_p}(1-k,\varepsilon_c),
        \qquad k\geq1.
\]
See \cite[pp.~230--231]{DeligneRibet} for the definition of \(\mathrm{L}_{L,S_p}(1-k,\varepsilon)\) in terms of partial zeta functions and its specialization to \(S_p\)-depleted Hecke \(L\)-values for finite-order characters.
For fixed \(c\) and \(k\geq1\), let
\[
        \mu_{c,k}(\varepsilon)
        :=
        \Delta_c(1-k,\varepsilon).
\]
A priori this defines only a \(\mathbb Q_p\)-valued distribution.  The
essential integrality assertion is
\cite[Theorem~(0.5), p.~231]{DeligneRibet}: the distribution
\(\mu_{c,k}\) is in fact a \(\mathbb Z_p\)-valued measure, and the
measures for different \(k\) satisfy
\[
        \mu_{c,k}
        =
        \kappa_L^{\,k-1}\mu_{c,1}.
\]
Here multiplication of a measure by \(\kappa_L^{\,k-1}\) means that
\[
        \int_{\mathcal G_L}f\,d
        \bigl(\kappa_L^{\,k-1}\mu_{c,1}\bigr)
        =
        \int_{\mathcal G_L}
        f(g)\kappa_L(g)^{k-1}\,d\mu_{c,1}(g).
\]
The general form of this assertion is Main Theorem~(8.2) of
\cite[p.~278]{DeligneRibet}.  Its proof ultimately rests on the
integrality of suitable linear combinations of constant terms of
Hilbert modular Eisenstein series, encoded in the ``Eisenstein
congruences'' of
\cite[Theorem~(8.4), pp.~278--280]{DeligneRibet}; the proof of
Main Theorem~(8.2) is completed on
\cite[pp.~282--283]{DeligneRibet}.

Following the normalization used by Serre, Deligne--Ribet then put
\[
        \lambda_c
        :=
        \kappa_L^{-1}\mu_{c,1}.
\]
Their Theorem~(0.5) gives
\begin{equation}
\label{eq:DR-lambda-c-interpolation}
        \int_{\mathcal G_L}
        \varepsilon(g)\kappa_L(g)^k
        \,d\lambda_c(g)
        =
        \Delta_c(1-k,\varepsilon)
\end{equation}
for every \(k\geq1\) and every locally constant \(\varepsilon\).
Thus \(\lambda_c\) is already an honest element of
\(\Lambda_L\), but it interpolates the \emph{modified} values
\(\Delta_c(1-k,\varepsilon)\), rather than the unmodified
\(L\)-values.

The passage from these measures to the \(p\)-adic zeta
\emph{pseudomeasure} removes the modifying factor \(1-c\).
To formulate this precisely, let $Q(\Lambda_L)$ denote the total quotient ring of \(\Lambda_L\), obtained by inverting
all non-zero-divisors of \(\Lambda_L\). Choose \(c\in\mathcal G_L\) such that \(1-c\) is a non-zero-divisor in
\(\Lambda_L\).  Deligne--Ribet define
\begin{equation}
\label{eq:DR-pseudomeasure-definition}
        \Theta_{L,S_p}
        :=
        \frac{\lambda_c}{1-c}
        \in
        Q(\Lambda_L).
\end{equation}
This is precisely the construction given immediately after
Theorem~(0.5) on
\cite[p.~232]{DeligneRibet}.

An element $\Theta\in Q(\Lambda_L)$
is called a \emph{pseudomeasure} if
\begin{equation}
\label{eq:def-pseudomeasure}
        (1-g)\Theta\in\Lambda_L
        \qquad
        \text{for every }g\in\mathcal G_L.
\end{equation}
Thus a pseudomeasure is allowed to have a denominator, but only a
denominator of augmentation type. Deligne--Ribet show that
the element \(\Theta_{L,S_p}\) in
\eqref{eq:DR-pseudomeasure-definition} is independent of the auxiliary
choice of \(c\): more precisely, if \(c'\in\mathcal G_L\), then
\begin{equation}
\label{eq:DR-pseudomeasure-relation}
        (1-c')\Theta_{L,S_p}
        =
        \lambda_{c'}
        \in\Lambda_L.
\end{equation}
This is the relation stated on
\cite[p.~232]{DeligneRibet}, and it immediately implies
\eqref{eq:def-pseudomeasure}.

The reason for introducing the pseudomeasure is that division by
\(1-c\) removes the auxiliary modification in
\(\Delta_c(1-k,\varepsilon)\).  Consequently, in the sense of
specialization in the total quotient ring, one has
\begin{equation}
\label{eq:DR-pseudomeasure-interpolation}
        \int_{\mathcal G_L}
        \varepsilon(g)\kappa_L(g)^k
        \,d\Theta_{L,S_p}(g)
        =
        \mathrm{L}_{L,S_p}(1-k,\varepsilon).
\end{equation}
This is the interpolation statement described by Deligne--Ribet
immediately after the definition of their pseudomeasure on
\cite[p.~232]{DeligneRibet}.  Specializing
\(\varepsilon\) to finite-order characters gives the usual
Deligne--Ribet \(p\)-adic \(L\)-functions.

\par We now pass from the Deligne--Ribet pseudomeasure to the integral
Iwasawa function that will be used below.  Fix an even integer
\(n\geq2\), and let
\begin{equation}\label{rhoLndefn}\rho_{L,n}:=\omega_L^n
\end{equation}
be the corresponding character of \(\Delta_L\).  Let
\(\mathcal O_{L,n}\) denote the valuation ring generated over
\(\mathbb Z_p\) by the values of \(\rho_{L,n}\), and put
\[
        \varepsilon_{L,n}
        :=
        \frac{1}{|\Delta_L|}
        \sum_{\delta\in\Delta_L}
        \rho_{L,n}(\delta)^{-1}\delta
        \in
        \mathcal O_{L,n}[\Delta_L].
\]
We impose the nonexceptional condition
\begin{equation}
\label{eq:totally-real-nonexceptional}
        \rho_{L,n}\neq1.
\end{equation}

\begin{lemma}
\label{lem:DR-nonexceptional-integral}
Assume \eqref{eq:totally-real-nonexceptional}.  Then the
\(\rho_{L,n}\)-component of the Deligne--Ribet pseudomeasure is integral:
\begin{equation}
\label{eq:DR-nonexceptional-integral}
        \varepsilon_{L,n}\Theta_{L,S_p}
        \in
        \varepsilon_{L,n}
        \mathcal O_{L,n}
        \llbracket\mathcal G_L\rrbracket
        \simeq
        \mathcal O_{L,n}\llbracket\Gamma_L\rrbracket.
\end{equation}
In particular, after choosing a topological generator of \(\Gamma_L\),
this component may be regarded as a power series in
\(\mathcal O_{L,n}\llbracket T\rrbracket\).
\end{lemma}

\begin{proof}
Since \(\Theta_{L,S_p}\) is a pseudomeasure, for every
\(g\in\mathcal G_L\) one has
\[
        (g-1)\Theta_{L,S_p}
        \in
        \mathbb Z_p\llbracket\mathcal G_L\rrbracket.
\]
Because \(\rho_{L,n}\neq1\), we may choose
\(\delta\in\Delta_L\) such that
\(\rho_{L,n}(\delta)\neq1\).  Applying \(\varepsilon_{L,n}\) to the preceding
relation with \(g=\delta\) gives
\[
        \bigl(\rho_{L,n}(\delta)-1\bigr)
        \varepsilon_{L,n}\Theta_{L,S_p}
        \in
        \varepsilon_{L,n}
        \mathcal O_{L,n}\llbracket\mathcal G_L\rrbracket,
\]
up to the equivalent inverse convention determined by the definition of
\(\varepsilon_{L,n}\).  The value \(\rho_{L,n}(\delta)\) has order prime to \(p\);
hence
\[
        \rho_{L,n}(\delta)-1
        \in
        \mathcal O_{L,n}^{\times}.
\]
Dividing by this unit proves \eqref{eq:DR-nonexceptional-integral}.

Thus the possible denominator of the Deligne--Ribet pseudomeasure
disappears on every nontrivial \(\Delta_L\)-component.  Only the trivial
component can remain nonintegral; this is the exceptional branch of the
\(p\)-adic zeta function.  Since the weight-\(n\) component is
nontrivial by \eqref{eq:totally-real-nonexceptional}, it is an honest
Iwasawa measure, and this is the component to which we apply the
cyclotomic norm theorem.
\end{proof}

Since \(\rho_{L,n}=\omega_L^n\) takes values in
\(\mathbb Z_p^\times\), there is a continuous
\(\mathbb Z_p\)-algebra homomorphism
\begin{equation}
\label{eq:rho-specialization-map}
        \operatorname{pr}_{\rho_{L,n}}:
        \mathbb Z_p\llbracket\mathcal G_L\rrbracket
        \longrightarrow
        \mathbb Z_p\llbracket\Gamma_L\rrbracket
\end{equation}
characterized on group-like elements by
\[
        \operatorname{pr}_{\rho_{L,n}}(\delta\gamma)
        =
        \rho_{L,n}(\delta)\gamma,
\]
where $\delta\in\Delta_L$ and $\gamma\in\Gamma_L$.
In other words, this map specializes the finite
\(\Delta_L\)-variable at the character \(\rho_{L,n}\) and leaves the
cyclotomic variable unchanged.  Since
\[
        \operatorname{pr}_{\rho_{L,n}}(\varepsilon_{L,n})=1,
\]
its restriction induces an isomorphism
\begin{equation}
\label{eq:rho-component-identification}
        \varepsilon_{L,n}
        \mathbb Z_p\llbracket\mathcal G_L\rrbracket
        \xrightarrow{\;\sim\;}
        \mathbb Z_p\llbracket\Gamma_L\rrbracket,
\end{equation}
defined by $\varepsilon_{L,n}\gamma\longmapsto\gamma$. By Lemma~\ref{lem:DR-nonexceptional-integral}, the element
\(\varepsilon_{L,n}\Theta_{L,S_p}\) lies in the source of
\eqref{eq:rho-component-identification}.  We therefore set
\[
        \Theta_{L,n}
        :=
        \operatorname{pr}_{\rho_{L,n}}
        \bigl(\varepsilon_{L,n}\Theta_{L,S_p}\bigr)
        \in
        \mathbb Z_p\llbracket\Gamma_L\rrbracket.
\]

It remains to incorporate the weight-\(n\) cyclotomic twist. Recall that $\kappa_L$ is given by \eqref{kappadefn} and write
\[
        \kappa_{L,\Gamma}
        :=
        \kappa_L|_{\Gamma_L}:
        \Gamma_L\longrightarrow1+p\mathbb Z_p.
\]
For every integer \(a\), the rule
\[
        \gamma\longmapsto
        \kappa_{L,\Gamma}(\gamma)^a\gamma,
        \qquad
        \gamma\in\Gamma_L,
\]
extends uniquely by continuity to a
\(\mathbb Z_p\)-algebra automorphism
\begin{equation}
\label{eq:cyclotomic-twist-map}
        \operatorname{Tw}_a:
        \mathbb Z_p\llbracket\Gamma_L\rrbracket
        \xrightarrow{\;\sim\;}
        \mathbb Z_p\llbracket\Gamma_L\rrbracket.
\end{equation}
Its inverse is \(\operatorname{Tw}_{-a}\). The twisting automorphism is given explicitly by
\[
        \operatorname{Tw}_a:
        G(T)
        \longmapsto
        G\bigl(u_L^a(1+T)-1\bigr).
\]
Thus twisting amounts simply to translating the cyclotomic variable by
the character \(\kappa_{L,\Gamma}^a\).

With the exponent dictated by the interpolation normalization, we define
\begin{equation}
\label{eq:DR-relative-Iwasawa-element}
        \mathscr G_{L,n}
        :=
        \operatorname{Tw}_n(\Theta_{L,n})
        \in
        \mathbb Z_p\llbracket\Gamma_L\rrbracket.
\end{equation}
Equivalently,
\[
        \mathscr G_{L,n}
        =
        \operatorname{Tw}_n
        \left(
        \operatorname{pr}_{\rho_{L,n}}
        \bigl(\varepsilon_{L,n}\Theta_{L,S_p}\bigr)
        \right).
\]
Hence the construction is the composition
\[
        \varepsilon_{L,n}\Theta_{L,S_p}
        \in
        \varepsilon_{L,n}\mathbb Z_p\llbracket\mathcal G_L\rrbracket
        \xrightarrow{\;\operatorname{pr}_{\rho_{L,n}}\;}
        \mathbb Z_p\llbracket\Gamma_L\rrbracket
        \xrightarrow{\;\operatorname{Tw}_n\;}
        \mathbb Z_p\llbracket\Gamma_L\rrbracket.
\]

Under
\eqref{eq:totally-real-iwasawa-power-series}, write
\begin{equation}\label{GLn}G_{L,n}(T)
        :=
        \iota_{\gamma_L}(\mathscr G_{L,n})
        \in
        \mathbb Z_p\llbracket T\rrbracket.
\end{equation}

The normalization has been chosen so that, for every finite-order
character
\[
        \chi:\Gamma_L
        \longrightarrow
        \overline{\mathbb Q}_p^{\,\times},
\]
one has
\begin{equation}
\label{eq:DR-relative-interpolation}
        G_{L,n}
        \bigl(\chi(\gamma_L)-1\bigr)
        =
        \mathrm{L}_{L,S_p}(1-n,\chi).
\end{equation}
\par We now apply the Artin formalism to the imprimitive zeta function. We have that
\begin{equation}
\label{eq:Artin-totally-real-p-depleted}
        \zeta_{L^{(m)},S_p}(s)
        =
        \prod_{\chi\in\Psi_{L,m}}
        \mathrm{L}_{L,S_p}(s,\chi),
\end{equation}
where
\[
        \zeta_{L^{(m)},S_p}(s)
        :=
        \zeta_{L^{(m)}}(s)
        \prod_{\mathfrak P\mid p}
        \left(
        1-(N\mathfrak P)^{-s}
        \right)
\]
is the Dedekind zeta function with all Euler factors above \(p\) removed.

Recall from \eqref{A_m and H_m} that, for
\(G(T)\in\mathbb Z_p\llbracket T\rrbracket\), we defined
\[
        A_m(G)
        :=
        \prod_{\zeta^{p^m}=1}G(\zeta-1).
\]
Thus \eqref{eq:DR-product-zeta} may be written as
\begin{equation}
\label{eq:DR-product-zeta}
\begin{aligned}
        \zeta_{L^{(m)},S_p}(1-n)
        &=
        \prod_{\chi\in\Psi_{L,m}}
        G_{L,n}\bigl(\chi(\gamma_L)-1\bigr)\\
        &=
        \prod_{\zeta^{p^m}=1}G_{L,n}(\zeta-1)\\
        &=
        A_m(G_{L,n}),
\end{aligned}
\end{equation}
where the second equality uses the fact that
\(\chi(\gamma_L)\), as \(\chi\) ranges over \(\Psi_{L,m}\), runs through
all \(p^m\)-th roots of unity.

Theorem \ref{thm:abstract-norm} therefore gives immediately
\begin{equation}
\label{eq:p-depleted-zeta-congruence}
        \zeta_{L^{(m+1)},S_p}(1-n)
        \equiv
        \zeta_{L^{(m)},S_p}(1-n)
        \pmod{p^{m+1}\mathbb Z_p}.
\end{equation}
For \(m\geq0\), put
\begin{equation}
\label{eq:E-m-n}
        E_{m,n}
        :=
        \prod_{\mathfrak P\mid p\text{ in }L^{(m)}}
        \left(
        1-(N\mathfrak P)^{n-1}
        \right).
\end{equation}
Then
\begin{equation}
\label{eq:p-depletion-full-zeta}
        \zeta_{L^{(m)},S_p}(1-n)
        =
        E_{m,n}\,
        \zeta_{L^{(m)}}(1-n).
\end{equation}
Since \(N\mathfrak P\) is divisible by \(p\) and \(n\geq2\), every factor
in \(E_{m,n}\) belongs to \(1+p\mathbb Z_p\).  In particular, $E_{m,n}\in\mathbb Z_p^\times$. Let \(\mathfrak p\) be a prime of \(L\) above \(p\), and let $D_{\mathfrak p}\subseteq\Gamma_L$ be its decomposition group in \(L_\infty/L\).  The local cyclotomic
extension is totally ramified after its finite intersection with
\(\mathrm{L}_{\mathfrak p}\), and hence \(D_{\mathfrak p}\) is an open subgroup
of \(\Gamma_L\).  Since \(\Gamma_L\simeq\mathbb Z_p\), there is a unique
integer \(d_{\mathfrak p}\geq0\) such that $D_{\mathfrak p}
        =
        \Gamma_L^{p^{d_{\mathfrak p}}}$. It follows that the number of primes of \(L^{(m)}\) above
\(\mathfrak p\) is
\[
        g_{\mathfrak p,m}:=
        [\Gamma_L:
        D_{\mathfrak p}\Gamma_L^{p^m}]=
        p^{\min\{m,d_{\mathfrak p}\}}.
\]
Moreover, the residue degree does not change in the cyclotomic direction,
so each of these primes has absolute norm $N\mathfrak p$.
Consequently from \eqref{eq:E-m-n}, we deduce that
\begin{equation}
\label{eq:Euler-product-relative}
        E_{m,n}
        =
        \prod_{\mathfrak p\mid p}
        \left(
        1-(N\mathfrak p)^{n-1}
        \right)^{
        p^{\min\{m,d_{\mathfrak p}\}}
        }.
\end{equation}

\begin{lemma}
\label{lem:Euler-factor-growing-congruence}
For every \(m\geq0\),
\[
        E_{m+1,n}
        \equiv
        E_{m,n}
        \pmod{p^{m+1}\mathbb Z_p}.
\]
\end{lemma}

\begin{proof}
Since $E_{m,n}$ is a unit in $\Z_p$, it suffices to show that 
\[
        \frac{E_{m+1,n}}{E_{m,n}}
        \equiv
        1
        \pmod{p^{m+1}\mathbb Z_p}.
\]
We compare the Euler factors occurring in \(E_{m+1,n}\) and
\(E_{m,n}\) prime by prime. Fix a prime \(\mathfrak p\mid p\), and set $a_{\mathfrak p}
        :=
        1-(N\mathfrak p)^{n-1}$. Since \(p\mid N\mathfrak p\), we have
\((N\mathfrak p)^{n-1}\in p\mathbb Z_p\), and therefore $a_{\mathfrak p}\in 1+p\mathbb Z_p$. In particular, \(a_{\mathfrak p}\) is a \(p\)-adic unit. Recall from \eqref{eq:Euler-product-relative} that the exponent of the \(\mathfrak p\)-factor in \(E_{m,n}\)
grows like \(p^m\) until the level \(d_{\mathfrak p}\), after which it
stabilizes. Thus there are two cases. If
\(m\geq d_{\mathfrak p}\), the exponent has already stabilized, so the
\(\mathfrak p\)-factor is the same in \(E_{m+1,n}\) and \(E_{m,n}\). Suppose now that \(m<d_{\mathfrak p}\). The exponent of
\(a_{\mathfrak p}\) changes from \(p^m\) at level \(m\) to
\(p^{m+1}\) at level \(m+1\).  Hence its contribution to the quotient is
\[
        a_{\mathfrak p}^{p^{m+1}-p^m}
        =
        a_{\mathfrak p}^{(p-1)p^m}.
\]Thus we deduce that 
\[
\frac{E_{m+1,n}}{E_{m,n}}
=
\prod_{\substack{\mathfrak p\mid p\\ m<d_{\mathfrak p}}}
a_{\mathfrak p}^{(p-1)p^m}.
\]
\par For any \(a\in1+p\mathbb Z_p\), with \(p\) odd, one has
\begin{equation}\label{eq:p-power-unit-congruence}
        a^{p^m}\equiv1\pmod{p^{m+1}\mathbb Z_p}.
\end{equation}
Indeed, writing \(a=1+pu\) with \(u\in\mathbb Z_p\), the standard
valuation estimate
\[
        v_p\bigl(a^{p^m}-1\bigr)
        \geq v_p(a-1)+m
        \geq m+1
\]
gives \eqref{eq:p-power-unit-congruence}. Applying this to \(a=a_{\mathfrak p}\), we obtain
\[
        a_{\mathfrak p}^{p^m}
        \equiv1\pmod{p^{m+1}\mathbb Z_p},
\]
and consequently
\[
        a_{\mathfrak p}^{(p-1)p^m}
        =
        \bigl(a_{\mathfrak p}^{p^m}\bigr)^{p-1}
        \equiv1\pmod{p^{m+1}\mathbb Z_p}.
\]
This proves that
\[
        E_{m+1,n}
        \equiv
        E_{m,n}
        \pmod{p^{m+1}\mathbb Z_p},
\]
as required.
\end{proof}

Combining the Deligne--Ribet interpolation with the cyclotomic norm
argument now gives the following form of the zeta-value theorem.

\begin{theorem}
\label{thm:totally-real-zeta-congruence}
Let \(L\) be a totally real number field and let \(p\) be an odd prime.
Write $L\cap\mathbb Q_\infty
        =
        \mathbb Q_{e_L}$ and $L^{(m)}
        =
        L\mathbb Q_{e_L+m}$. Let \(n\geq2\) be even and suppose that the corresponding
Teichm\"uller branch is nonexceptional, that is, $\omega_L^n\neq1$.
Then $\zeta_{L^{(m)}}(1-n)
        \in
        \mathbb Z_p$ for every \(m\geq0\), and
\begin{equation}
\label{eq:totally-real-zeta-main-congruence}
        \zeta_{L^{(m+1)}}(1-n)
        \equiv
        \zeta_{L^{(m)}}(1-n)
        \pmod{p^{m+1}\mathbb Z_p}.
\end{equation}
\end{theorem}

\begin{proof}
By Lemma~\ref{lem:DR-nonexceptional-integral}, 
\(G_{L,n}(T)\) belongs to
\(\mathbb Z_p\llbracket T\rrbracket\). Recall that \eqref{eq:p-depleted-zeta-congruence} states that
\[
        \zeta_{L^{(m+1)},S_p}(1-n)
        \equiv
        \zeta_{L^{(m)},S_p}(1-n)
        \pmod{p^{m+1}}.
\]
\noindent Using
\eqref{eq:p-depletion-full-zeta}, the preceding congruence becomes
\[
        E_{m+1,n}
        \zeta_{L^{(m+1)}}(1-n)
        \equiv
        E_{m,n}
        \zeta_{L^{(m)}}(1-n)
        \pmod{p^{m+1}}.
\]
By Lemma~\ref{lem:Euler-factor-growing-congruence},
\[
        E_{m+1,n}
        \equiv
        E_{m,n}
        \pmod{p^{m+1}}.
\]
Since the zeta values are \(p\)-integral, $E_{m,n}$ and \(E_{m+1,n}\) are units, we
deduce
\[
        \zeta_{L^{(m+1)}}(1-n)
        \equiv
        \zeta_{L^{(m)}}(1-n)
        \pmod{p^{m+1}},
\]
as desired.
\end{proof}

\section{Applications to higher \(K\)-groups and arithmetic Euler characteristics}\label{s 4}

\label{sec:applications-k-groups-euler-characteristics}

\par Throughout this section, let \(L\) be a finite totally real number
field.  Let \(\mathbb Q_\infty/\mathbb Q\) denote the cyclotomic
\(\mathbb Z_p\)-extension, and let \(e_L\geq0\) be determined by
\[
        L\cap\mathbb Q_\infty=\mathbb Q_{e_L}.
\]
As in Section~\ref{s 3}, for \(m\geq0\) we put
\[
        L^{(m)}
        :=
        L\mathbb Q_{e_L+m}.
\]
Then \(L^{(0)}=L\),
\[
        \operatorname{Gal}(L^{(m)}/L)
        \simeq
        \mathbb Z/p^m\mathbb Z,
\]
and every \(L^{(m)}\) is totally real.  Thus the indexing used throughout
this section does not require any linear-disjointness hypothesis between
\(L\) and \(\mathbb Q_\infty\).

\subsection{The Birch--Tate--Lichtenbaum formula}
\label{subsec:BTL-formula}

\par Let \(F\) be a number field and let \(\mathcal O_F\) denote its
ring of integers. Quillen proved that the groups
\(K_i(\mathcal O_F)\) are finitely generated, and Borel computed their
ranks; see \cite{Quillen1973,Borel1977}. It follows that
\(K_{2j}(\mathcal O_F)\) is finite for every \(j\geq 1\). The
finiteness of \(K_2(\mathcal O_F)\) had previously been established by
Garland \cite{Garland1971}.

\par To state the formula for the orders of these groups, let
\(G_F=\operatorname{Gal}(\overline F/F)\), let \(\mu_\infty\) be the
group of all roots of unity in \(\overline F\), and, for \(j\geq 1\),
put
\begin{equation}
        w_j(F)
        :=
        \#H^0\bigl(G_F,\mathbb Q/\mathbb Z(j)\bigr)
        =
        \#\bigl(\mu_\infty^{\otimes j}\bigr)^{G_F}.
        \label{eq:def-wj}
\end{equation}
The group in \eqref{eq:def-wj} is finite. Note that $w_1(F)$ is the number of roots of unity contained in $F$. 
\par We now recall the Brich--Tate--Lichtenbaum formula below which relates zeta values to $K$-groups of $\cO_F$ for a totally real abelian number field $F$.

\begin{theorem}[Birch--Tate--Lichtenbaum]
\label{thm:BTL-exact}
Let \(F/\mathbb Q\) be a totally real abelian extension, let
\(d=[F:\mathbb Q]\), and let \(n\geq 2\) be even. Then
\begin{equation}
        \#K_{2n-2}(\mathcal O_F)
        =
        \begin{cases}
        (-1)^{d}w_n(F)\zeta_F(1-n),
            & n\equiv 2\pmod 4,\\[4pt]
        2^{-d}w_n(F)\zeta_F(1-n),
            & n\equiv 0\pmod 4.
        \end{cases}
        \label{eq:BTL-exact}
\end{equation}
In particular, for every odd prime \(p\),
\begin{equation}
        v_p\!\left(\#K_{2n-2}(\mathcal O_F)\right)
        =
        v_p\!\left(w_n(F)\zeta_F(1-n)\right).
        \label{eq:BTL-odd-primary}
\end{equation}
\end{theorem}

\par The case \(n=2\) is the Birch--Tate conjecture, and the general
formula was conjectured by Lichtenbaum
\cite{Lichtenbaum1972,Lichtenbaum1973}.  The relation between special
zeta values and algebraic $K$-groups was also developed in the work of
Coates and Coates--Sinnott; see \cite{CoatesK2,CoatesSinnott}.  For a
broader discussion of $K$-theory and zeta values, see
\cite{SouleKTheory}.

\subsection{The \(p\)-primary root-of-unity factor}
\label{subsec:p-primary-w-factor}

\par The factor \(w_n(L^{(m)})\) must be understood before the
zeta-value congruences can be translated into statements about
\(K\)-groups. Put
\begin{equation}
        d_{L,p}
        :=
        [L(\mu_p):L].
        \label{eq:def-dLp}
\end{equation}
Since \(L(\mu_p)/L\) is cyclic of degree dividing \(p-1\), the integer
\(d_{L,p}\) is prime to \(p\).  It records the part of the tame
\(p\)-cyclotomic extension which is not already contained in \(L\).

\begin{lemma}
\label{lem:w-factor-equivalences}
Let \(L\) be a totally real number field and let \(n\geq1\). The following
conditions are equivalent:
\begin{enumerate}
\item \(d_{L,p}\mid n\);
\item \(p\mid w_n(L^{(m)})\) for some \(m\geq0\);
\item \(p\mid w_n(L^{(m)})\) for every \(m\geq0\).
\end{enumerate}
Consequently, if \(d_{L,p}\nmid n\), then $v_p\bigl(w_n(L^{(m)})\bigr)=0$ for $m\geq0$.
\end{lemma}
\begin{proof}
Let \(E=L\cap\mathbb Q(\mu_p)\). Restriction identifies
\[
        \operatorname{Gal}(L(\mu_p)/L)
        \simeq
        \operatorname{Gal}(\mathbb Q(\mu_p)/E),
\]
and this is a cyclic subgroup of
\((\mathbb Z/p\mathbb Z)^\times\) of order \(d_{L,p}\). The
Teichm\"uller character is faithful on
\(\operatorname{Gal}(\mathbb Q(\mu_p)/\mathbb Q)\). It follows that
\(\omega^n\) is trivial on
\(\operatorname{Gal}(\mathbb Q(\mu_p)/E)\) if and only if
\(d_{L,p}\mid n\).

\par Since \(L^{(m)}/L\) is a \(p\)-extension and
\([L(\mu_p):L]=d_{L,p}\) is prime to \(p\), one has
\begin{equation}
        L^{(m)}\cap L(\mu_p)=L
        \qquad\text{and}\qquad
        \operatorname{Gal}(L^{(m)}(\mu_p)/L^{(m)})
        \simeq
        \operatorname{Gal}(L(\mu_p)/L).
        \label{eq:linear-disjoint-mup}
\end{equation}
Now \(p\mid w_n(L^{(m)})\) if and only if
\((\mu_p^{\otimes n})^{G_{L^{(m)}}}\neq 0\). The action of
\(G_{L^{(m)}}\) on this one-dimensional \(\mathbb F_p\)-space is given
by the \(n\)-th power of the mod-\(p\) cyclotomic character. In view of
\eqref{eq:linear-disjoint-mup}, this character is trivial if and only
if \(d_{L,p}\mid n\). This proves that the conditions are equivalent.
\end{proof}
\begin{remark}
Let $\omega=\omega_L$ be the Teichm\"uller character as in \eqref{omegaL}. If \(X_L\) denotes the group of
Dirichlet characters associated with \(L\), then
\[
        d_{L,p}\mid n
        \quad\text{if and only if}\quad
        \omega^{-n}\in X_L.
\]
Indeed, writing \(E=L\cap\mathbb Q(\mu_p)\), the character
\(\omega^{-n}\) belongs to \(X_L\) precisely when it is trivial on
\(\operatorname{Gal}(\mathbb Q(\mu_p)/E)\), whose order is
\(d_{L,p}\).
\end{remark}

\subsection{Stability of \(p\)-torsion in higher even \(K\)-groups}
\label{subsec:K-group-stability}
\par Assume that $L$ is an abelian totally real extension of $\Q$. For an even integer \(n\geq 2\), set
\begin{equation}
        \kappa_m(n)
        :=
        \#K_{2n-2}(\mathcal O_{L^{(m)}}).
        \label{eq:def-kappa}
\end{equation}

\begin{theorem}
\label{thm:K-group-congruence-application}
Let \(n\geq 2\) be even and assume that
\(d_{L,p}\nmid n\). Then, for every \(m\geq 0\),
\begin{equation}
        v_p\bigl(\kappa_m(n)\bigr)
        =
        v_p\bigl(\zeta_{L^{(m)}}(1-n)\bigr).
        \label{eq:kappa-zeta-valuation}
\end{equation}
Moreover,
\begin{equation}
\begin{split}
        \min\!\left\{
        v_p\bigl(\kappa_{m+1}(n)\bigr),m+1
        \right\}
        &=
        \min\!\left\{
        v_p\bigl(\kappa_m(n)\bigr),m+1
        \right\}.
\end{split}
        \label{eq:truncated-valuation-stability}
\end{equation}
Equivalently, for every integer \(r\) satisfying
\(1\leq r\leq m+1\), one has
\begin{equation}
        p^r\mid\kappa_{m+1}(n)
        \quad\text{if and only if}\quad
        p^r\mid\kappa_m(n).
        \label{eq:pr-divisibility-stability}
\end{equation}
\end{theorem}

\begin{proof}
It follows from our hypothesis $d_{L,p}\nmid n$ that \(w_n(L^{(m)})\) is a \(p\)-adic unit for every \(m\).
Equation \eqref{eq:kappa-zeta-valuation} therefore follows from the
odd-primary part \eqref{eq:BTL-odd-primary} of the
Birch--Tate--Lichtenbaum formula.

\par By Theorem \ref{thm:totally-real-zeta-congruence}, the two
\(p\)-adic integers \(\zeta_{L^{(m+1)}}(1-n)\) and
\(\zeta_{L^{(m)}}(1-n)\) are congruent modulo \(p^{m+1}\). If
\(x,y\in\mathbb Z_p\) satisfy \(x\equiv y\pmod{p^a}\), then
\(\min\{v_p(x),a\}=\min\{v_p(y),a\}\). Applying this observation with
\(a=m+1\), and then using \eqref{eq:kappa-zeta-valuation}, gives
\eqref{eq:truncated-valuation-stability}.
The last assertion \eqref{eq:pr-divisibility-stability} is an immediate
reformulation.
\end{proof}

\begin{corollary}
\label{cor:K-group-p-torsion-dichotomy}
Under the hypotheses of Theorem
\ref{thm:K-group-congruence-application}, the following are
equivalent:
\begin{enumerate}
\item \(K_{2n-2}(\mathcal O_L)[p^\infty]=0\);
\item \(\zeta_L(1-n)\in\mathbb Z_p^\times\);
\item \(K_{2n-2}(\mathcal O_{L^{(m)}})[p^\infty]=0\) for some
\(m\geq 0\);
\item \(K_{2n-2}(\mathcal O_{L^{(m)}})[p^\infty]=0\) for every
\(m\geq 0\).
\end{enumerate}
If these equivalent conditions fail, then
\(K_{2n-2}(\mathcal O_{L^{(m)}})[p^\infty]\neq 0\) for every
\(m\geq 0\).
\end{corollary}

\begin{proof}
By \eqref{eq:kappa-zeta-valuation}, the \(p\)-primary subgroup in
question is trivial if and only if the corresponding zeta value is a
\(p\)-adic unit. The congruence
\[
        \zeta_{L^{(m+1)}}(1-n)
        \equiv
        \zeta_{L^{(m)}}(1-n)
        \pmod p
\]
shows that \(p\)-divisibility is unchanged from one layer to the next.
\end{proof}

\begin{corollary}[Quantitative \(K\)-group growth]
\label{cor:K-growth}
Assume the hypotheses of
Theorem~\ref{thm:K-group-congruence-application}. If $K_{2n-2}(\mathcal O_L)[p^\infty]\neq0$,
then, for every \(m\geq0\),
\[
        v_p\!\left(
        \#K_{2n-2}(\mathcal O_{L^{(m)}})
        \right)
        \geq m+1.
\]
\end{corollary}

\begin{proof}
By Theorem~\ref{thm:K-group-congruence-application}, applied at
\(m=0\), we have
\[
        v_p\!\left(\#K_{2n-2}(\mathcal O_L)\right)
        =
        v_p\bigl(\zeta_L(1-n)\bigr),
\]
and consequently,
\begin{equation}
        v_p\bigl(\zeta_L(1-n)\bigr)\geq1.
        \label{eq:zeta-base-nonunit}
\end{equation}

Recall that, for the normalization under consideration, the
Deligne--Ribet power series $G_{L,n}(T)\in\mathbb Z_p\llbracket T\rrbracket$ (as defined in \eqref{GLn}) satisfies
\begin{equation}
        \zeta_{L^{(m)}}(1-n)
        =
        u_m A_m(G_{L,n}),
        \qquad
        u_m\in\mathbb Z_p^\times,
        \label{eq:zeta-Am-unit-factor-proof}
\end{equation}
where
\begin{equation}
        A_m(G_{L,n})
        :=
        \prod_{\zeta^{p^m}=1}
        G_{L,n}(\zeta-1).
        \label{eq:Am-GLn-proof}
\end{equation}
First note that
\[
\begin{aligned}
        \zeta_L(1-n)
        &=
        u_0A_0(G_{L,n})\\
        &=
        u_0
        \prod_{\zeta=1}
        G_{L,n}(\zeta-1)\\
        &=
        u_0G_{L,n}(0).
\end{aligned}
\]
Because \(u_0\in\mathbb Z_p^\times\), we have that
\[
\begin{aligned}
        v_p\bigl(G_{L,n}(0)\bigr)
        &=
        v_p\bigl(\zeta_L(1-n)\bigr)-v_p(u_0)\\
        &=
        v_p\bigl(\zeta_L(1-n)\bigr)\\
        &\geq1.
\end{aligned}
\]
For
\(j\geq0\), let
\[
        H_j(G_{L,n})
        :=
        \prod_{\substack{\zeta^{p^{j+1}}=1\\
                         \zeta^{p^j}\neq1}}
        G_{L,n}(\zeta-1),
\]
the product being taken over the primitive \(p^{j+1}\)-st roots of
unity.  Since
\[
        R_{p^{j+1}}
        =
        R_{p^j}\sqcup P_{p^{j+1}},
\]
we have
\[
        A_{j+1}(G_{L,n})
        =
        A_j(G_{L,n})H_j(G_{L,n}).
\]
Iterating this identity and using
\[
        A_0(G_{L,n})=G_{L,n}(0)
\]
gives
\begin{equation}
        A_m(G_{L,n})
        =
        G_{L,n}(0)
        \prod_{j=0}^{m-1}H_j(G_{L,n}).
        \label{eq:Am-factorization-K-growth}
\end{equation}
The second assertion of
Lemma~\ref{lem:cyclotomic-norm-unit} applies to \(G_{L,n}\), and in particular
\[
        v_p\bigl(H_j(G_{L,n})\bigr)\geq1
\]
for $j\geq 0$. Taking \(p\)-adic valuations in
\eqref{eq:Am-factorization-K-growth}, we therefore obtain
\[
\begin{aligned}
        v_p\bigl(A_m(G_{L,n})\bigr)
        &=
        v_p\bigl(G_{L,n}(0)\bigr)
        +
        \sum_{j=0}^{m-1}
        v_p\bigl(H_j(G_{L,n})\bigr)\\
        &\geq
        1+\sum_{j=0}^{m-1}1\\
        &=
        m+1.
\end{aligned}
\]
Thus for $m\geq 0$, we have that
\begin{equation}
        v_p\bigl(A_m(G_{L,n})\bigr)\geq m+1.
        \label{eq:Am-growth-K}
\end{equation}

Returning to \eqref{eq:zeta-Am-unit-factor-proof}, since
\(u_m\in\mathbb Z_p^\times\), we have \(v_p(u_m)=0\), and hence
\[
\begin{aligned}
        v_p\bigl(\zeta_{L^{(m)}}(1-n)\bigr)
        &=
        v_p\bigl(u_mA_m(G_{L,n})\bigr)\\
        &=
        v_p(u_m)
        +
        v_p\bigl(A_m(G_{L,n})\bigr)\\
        &=
        v_p\bigl(A_m(G_{L,n})\bigr)\\
        &\geq m+1.
\end{aligned}
\]
Finally, the hypotheses of
Theorem~\ref{thm:K-group-congruence-application} imply, by
Lemma~\ref{lem:w-factor-equivalences}, that $d_{L,p}\nmid n$,
and consequently
\[
        v_p\bigl(w_n(L^{(m)})\bigr)=0
        \qquad(m\geq0).
\]
Applying the odd-primary Birch--Tate--Lichtenbaum formula
\eqref{eq:BTL-odd-primary} to the totally real abelian field
\(L^{(m)}\), we obtain
\[
\begin{aligned}
        v_p\!\left(
        \#K_{2n-2}(\mathcal O_{L^{(m)}})
        \right)
        &=
        v_p\!\left(
        w_n(L^{(m)})
        \zeta_{L^{(m)}}(1-n)
        \right)\\
        &=
        v_p\bigl(w_n(L^{(m)})\bigr)
        +
        v_p\bigl(\zeta_{L^{(m)}}(1-n)\bigr)\\
        &=
        v_p\bigl(\zeta_{L^{(m)}}(1-n)\bigr)\\
        &\geq m+1.
\end{aligned}
\]
This proves the result.
\end{proof}
Recall that the Bernoulli numbers $B_n\in\Q$ are defined by
\[
\frac{t}{e^t-1}=\sum_{n\geq 0}B_n\frac{t^n}{n!},
\]
and satisfy $\zeta_{\Q}(1-n)=-B_n/n$ for $n\geq 1$.
\begin{corollary}
\label{cor:K-groups-rational-base}
Let \(L=\mathbb Q\), let \(n\geq 2\) be even, and assume that
\((p-1)\nmid n\). If $p\nmid \frac{B_n}{n}$, then
\begin{equation}
        K_{2n-2}(\mathcal O_{\mathbb Q_m})[p^\infty]=0
        \label{eq:K-vanishing-Qm}
\end{equation}
for $m\geq 0$. Conversely, if \(p\mid B_n/n\), then the group in
\eqref{eq:K-vanishing-Qm} is nonzero for every \(m\geq 0\).
\end{corollary}

\begin{proof}
For \(L=\mathbb Q\), one has
\(\zeta_{\mathbb Q}(1-n)=-B_n/n\), while
\(d_{\mathbb Q,p}=p-1\), so the nonexceptional condition is equivalent
to \((p-1)\nmid n\). The assertion follows from Corollary
\ref{cor:K-group-p-torsion-dichotomy}.
\end{proof}

\subsection{Generalized Bernoulli numbers}
\label{subsec:generalized-Bernoulli-applications}

\par In this section, we assume that \(L/\mathbb Q\) is a finite totally real
abelian extension.  For a finite abelian extension \(F/\mathbb Q\),
let \(X_F\) denote the group of Dirichlet characters associated with
\(F/\mathbb Q\).

\par Fix an embedding
\(\overline{\mathbb Q}\hookrightarrow\overline{\mathbb Q}_p\).
For a Dirichlet character \(\chi\) of conductor \(f_\chi\), its
generalized Bernoulli numbers are defined by
\begin{equation}
        \sum_{a=1}^{f_\chi}
        \chi(a)\frac{te^{at}}{e^{f_\chi t}-1}
        =
        \sum_{j=0}^{\infty}
        B_{j,\chi}\frac{t^j}{j!}.
        \label{eq:generalized-Bernoulli-generating-series}
\end{equation}
If \(n\geq 1\) and \(\chi(-1)=(-1)^n\), then
\begin{equation}\label{eq:Dirichlet-special-value-Bernoulli}
        L(1-n,\chi)
        =
        -\frac{B_{n,\chi}}{n}.
\end{equation}
See \cite[Chapter~4, \S4.1]{Washington1997} for these definitions and
the special-value formula \eqref{eq:Dirichlet-special-value-Bernoulli}.

\par Since \(L^{(m)}\) is totally real, every character in
\(X_{L^{(m)}}\) is even. Thus
\eqref{eq:Dirichlet-special-value-Bernoulli} applies to all characters
in \(X_{L^{(m)}}\) when \(n\) is even. The power-series construction
of the preceding section also gives the required \(p\)-integrality of
the individual values, rather than merely that of their product.

\begin{proposition}
\label{prop:Bernoulli-integrality}
Let \(n\geq 2\) be even and assume that
\(\omega^{-n}\notin X_L\). Then, for every \(m\geq 0\) and every
\(\chi\in X_{L^{(m)}}\),
\begin{equation}
        \frac{B_{n,\chi}}{n}
        \in
        \mathcal O_{\overline{\mathbb Q}_p}.
        \label{eq:Bernoulli-integrality}
\end{equation}
\end{proposition}

\begin{proof}
By Lemma~\ref{lem:w-factor-equivalences}, the hypothesis
\(\omega^{-n}\notin X_L\) is equivalent to \(d_{L,p}\nmid n\).
Since \(L^{(m)}/L\) is a \(p\)-extension whereas
\(L(\mu_p)/L\) has degree prime to \(p\),
\eqref{eq:linear-disjoint-mup} implies that
\[
        \omega^{-n}\notin X_{L^{(m)}}
        \qquad(m\geq0).
\]
Thus, for every \(\chi\in X_{L^{(m)}}\), the character
\(\chi\omega^n\) is nontrivial.  Since \(L^{(m)}\) is totally real and
\(n\) is even, this character is also even.  By Iwasawa's
power-series construction
\cite[Theorem~7.10]{Washington1997}, the Kubota--Leopoldt
\(p\)-adic \(L\)-function attached to the nontrivial character
\(\chi\omega^n\) is represented by an Iwasawa power series with
coefficients in
\(\mathcal O_{\overline{\mathbb Q}_p}\).  In particular,
\[
        L_p(1-n,\chi\omega^n)
        \in
        \mathcal O_{\overline{\mathbb Q}_p}.
\]
On the other hand, the interpolation formula
\cite[Theorem~5.11]{Washington1997} gives
\[
\begin{aligned}
        L_p(1-n,\chi\omega^n)
        &=
        \left(
        1-(\chi\omega^n)\omega^{-n}(p)p^{n-1}
        \right)
        L\left(1-n,(\chi\omega^n)\omega^{-n}\right)\\
        &=
        \left(1-\chi(p)p^{n-1}\right)L(1-n,\chi).
\end{aligned}
\]
Consequently,
\[
        \left(1-\chi(p)p^{n-1}\right)L(1-n,\chi)
        \in
        \mathcal O_{\overline{\mathbb Q}_p}.
\]
Since \(n\geq2\), the factor $1-\chi(p)p^{n-1}$
is a \(p\)-adic unit, it follows that
\[
        L(1-n,\chi)
        \in
        \mathcal O_{\overline{\mathbb Q}_p}.
\]
Finally, by \eqref{eq:Dirichlet-special-value-Bernoulli},
\[
        \frac{B_{n,\chi}}{n}
        =
        -L(1-n,\chi)
        \in
        \mathcal O_{\overline{\mathbb Q}_p},
\]
as required.
\end{proof}

\begin{theorem}
\label{thm:Bernoulli-unit-criterion}
Let \(n\geq 2\) be even and assume that
\(\omega^{-n}\notin X_L\). The following conditions are equivalent:
\begin{enumerate}
\item \(\zeta_L(1-n)\in\mathbb Z_p^\times\);
\item \(\zeta_{L^{(m)}}(1-n)\in\mathbb Z_p^\times\) for some
\(m\geq 0\);
\item \(\zeta_{L^{(m)}}(1-n)\in\mathbb Z_p^\times\) for every
\(m\geq 0\);
\item for some \(m\geq 0\), the quotient \(B_{n,\chi}/n\) is a
\(p\)-adic unit for every \(\chi\in X_{L^{(m)}}\);
\item for every \(m\geq 0\), the quotient \(B_{n,\chi}/n\) is a
\(p\)-adic unit for every \(\chi\in X_{L^{(m)}}\).
\end{enumerate}
\end{theorem}

\begin{proof}
The equivalence of the first three conditions follows from the
zeta-value congruence modulo \(p\). For a fixed \(m\), Artin
formalism and \eqref{eq:Dirichlet-special-value-Bernoulli} give
\begin{equation}
\begin{split}
        \zeta_{L^{(m)}}(1-n)
        &=
        \prod_{\chi\in X_{L^{(m)}}}L(1-n,\chi)\\
        &=
        (-1)^{[L^{(m)}:\mathbb Q]}
        \prod_{\chi\in X_{L^{(m)}}}\frac{B_{n,\chi}}{n}.
\end{split}
        \label{eq:zeta-product-Bernoulli}
\end{equation}
Every factor in the last product is \(p\)-integral by Proposition
\ref{prop:Bernoulli-integrality}. Hence the product is a \(p\)-adic
unit if and only if every factor is a \(p\)-adic unit. This proves the
remaining equivalences.
\end{proof}

\subsection{Euler characteristics of symplectic arithmetic groups}
\label{subsec:symplectic-Euler-characteristics}

\par We now return to an arbitrary finite totally real number field \(L\)
as at the beginning of the section.

\par Let \(F\) be totally real and let \(g\geq 1\). We write
\(\chi_{\mathrm{EP}}\bigl(\operatorname{Sp}_{2g}(\mathcal O_F)\bigr)\)
for the virtual Euler--Poincar\'e characteristic of the arithmetic
group \(\operatorname{Sp}_{2g}(\mathcal O_F)\). Harder's
Gauss--Bonnet formula gives
\begin{equation}
        \chi_{\mathrm{EP}}
        \bigl(\operatorname{Sp}_{2g}(\mathcal O_F)\bigr)
        =
        \prod_{j=1}^{g}\zeta_F(1-2j).
        \label{eq:Harder-Sp}
\end{equation}
Indeed, the fundamental degrees of the Weyl group of type \(C_g\) are
\(2,4,\ldots,2g\), and the Weyl-group constant in Harder's general
formula is equal to one for the split symplectic group. See
\cite[\S2.2, p.~453]{Harder1971}, especially the formula obtained
after applying the functional equation.

\begin{theorem}
\label{thm:Euler-characteristic-congruence}
Let \(g\geq 1\), and assume that $d_{L,p}\nmid 2j$ for $1\leq j\leq g$. Then, for every \(m\geq 0\), $\chi_{\mathrm{EP}}
        \bigl(\operatorname{Sp}_{2g}(\mathcal O_{L^{(m)}})\bigr)
        \in\mathbb Z_p$, and
\begin{equation}
\begin{split}
        \chi_{\mathrm{EP}}
        \bigl(\operatorname{Sp}_{2g}(\mathcal O_{L^{(m+1)}})\bigr)
        &\equiv
        \chi_{\mathrm{EP}}
        \bigl(\operatorname{Sp}_{2g}(\mathcal O_{L^{(m)}})\bigr)
        \pmod{p^{m+1}\mathbb Z_p}.
\end{split}
        \label{eq:Euler-congruence}
\end{equation}
\end{theorem}

\begin{proof}
For every \(1\leq j\leq g\), the condition
\(d_{L,p}\nmid2j\) is precisely the nonexceptional condition for the
weight-\(2j\) branch.  Theorem
\ref{thm:totally-real-zeta-congruence} therefore gives
\begin{equation}
        \zeta_{L^{(m+1)}}(1-2j)
        \equiv
        \zeta_{L^{(m)}}(1-2j)
        \pmod{p^{m+1}\mathbb Z_p},
        \label{eq:zeta-congruences-Harder}
\end{equation}
and both values belong to \(\mathbb Z_p\). Multiplying the congruences
in \eqref{eq:zeta-congruences-Harder} and invoking Harder's formula
\eqref{eq:Harder-Sp} proves the result.
\end{proof}

\section{Algebraic modular \(L\)-values in cyclotomic fields}
\label{sec:Mazur-Tate}

\par Let $f(\tau)
        =
        \sum_{r\geq1}a_r(f)e^{2\pi i r\tau}$ be a normalized cuspidal newform of weight \(k>2\), trivial
nebentype and level \(\Gamma_0(N)\), and let \(p\nmid N\) be an odd
prime.  For \(n\geq0\), put $K_n:=\Q(\mu_{p^{n+1}})$. The purpose of this section is to compare the algebraic special values $\mathrm{L}_{K_n}^{\operatorname{alg}}(f,1)$ as \(n\) varies through the cyclotomic tower. By Artin formalism, \(\mathrm{L}_{K_n}(f,s)\) is the product of the Dirichlet
twists of \(f\) corresponding to the characters of
\(\Gal(K_n/\Q)\).  In passing from \(K_{n-1}\) to \(K_n\), the new
characters are precisely the primitive Dirichlet characters of
conductor \(p^{n+1}\).  The corresponding product of twisted
\(L\)-values is encoded by the primitive character evaluations of a
Mazur--Tate element.  We shall use this to first obtain first a congruence
over an arbitrary coefficient ring \(\mathcal O\), and then a stronger
stability result akin to Theorem \ref{lthm a} when its residue field is \(\F_p\).

\subsection{Modular symbols and Mazur--Tate elements}

\par We begin by recalling the definition and properties of modular symbols and their relation to $L$-values of modular forms. For a comprehensive background on modular symbols, see
\cite{Manin1973,AshStevens,PollackWeston}.  Their use in the
construction of \(p\)-adic \(L\)-functions goes back to
Mazur--Swinnerton-Dyer and Amice--V\'elu
\cite{MazurSwinnertonDyer,AmiceVelu}; see also
\cite{PollackStevens}.

\par Set $g:=k-2$ and for a commutative ring \(R\) let $V_g(R):=\operatorname{Sym}^g(R^2)$ viewed as the space of homogeneous polynomials of degree \(g\) in
\(X,Y\).  If
\[
        \gamma=
        \begin{pmatrix}
        a&b\\ c&d
        \end{pmatrix},
\]
write
\[
        \gamma^*
        :=
        \begin{pmatrix}
        d&-b\\ -c&a
        \end{pmatrix}
\]
and equip \(V_g(R)\) with the right action
\[
        (P\mid\gamma)(X,Y)
        :=
        P\bigl((X,Y)\gamma^*\bigr)
        =
        P(dX-cY,-bX+aY).
\]

\par A modular symbol of weight \(k\) and level \(\Gamma_0(N)\) with
coefficients in \(R\) is an element of
\[
        H_c^1(\Gamma_0(N),V_g(R)).
\]
We use the canonical Hecke-equivariant identification
\begin{equation}
\label{basic mod symbol iso}
        H_c^1(\Gamma_0(N),V_g(R))
        \simeq
        \operatorname{Hom}_{\Gamma_0(N)}
        \left(
        \operatorname{Div}^0(\mathbb{P}^1(\Q)),
        V_g(R)
        \right),
\end{equation}
under which a modular symbol is represented by an additive map
\[
        \varphi:
        \operatorname{Div}^0(\mathbb{P}^1(\Q))
        \longrightarrow
        V_g(R)
\]
satisfying
\[
        \varphi(\gamma D)\mid\gamma
        =
        \varphi(D)
\]
for $\gamma\in\Gamma_0(N)$, cf. \cite[Proposition~4.2]{AshStevens}.

\par Associated with \(f\) is the complex modular symbol
\begin{equation}\label{complexmodsym}\xi_f(\{r\}-\{s\})
        =
        2\pi i
        \int_s^r
        f(z)(zX+Y)^{k-2}\,dz.
\end{equation}
It is a Hecke eigensymbol with the same Hecke eigenvalues as \(f\).
Letting $\iota
        =
        \begin{pmatrix}
        -1&0\\0&1
        \end{pmatrix}$, the involution induced by \(\iota\) gives a decomposition
\[
        \xi_f=\xi_f^++\xi_f^-.
\]

\par Fix an embedding $j:\overline{\Q}\hookrightarrow\overline{\Q}_p$. Let \(K_f\) be the Hecke field of \(f\), let \(\mathcal K\) be its
completion at the prime determined by \(j\), and let $O=\mathcal O_{\mathcal K}$, $\mathfrak m=\mathfrak m_{\mathcal K}$ be its maximal ideal and $\kappa=\mathcal O/\mathfrak m$ denote the residue field.
We write $a_p:=a_p(f)$.

\par We choose cohomological periods \(\Omega_f^\pm\), in the sense of
\cite[Definition~2.1]{PollackWeston}, such that
\[
        \varphi_f^\pm
        :=
        \frac{\xi_f^\pm}{\Omega_f^\pm}
        \in
        H_c^1
        \bigl(
        \Gamma_0(N),V_{k-2}(\mathcal O)
        \bigr).
\]
Such periods are unique up to multiplication by \(p\)-adic units.
Set
\[
        \varphi_f:=\varphi_f^++\varphi_f^-.
\]
The integrality of \(\varphi_f^\pm\) is the reason that the algebraic
\(L\)-values appearing below are \(p\)-integral.

\par We set $\zeta_k:=\exp\left(\frac{2\pi i}{k}\right)$. For \(m\geq1\), define the Mazur--Tate element of conductor
\(p^m\) by
\begin{equation}
\label{eq:MT-element}
        \vartheta_m(f)
        :=
        \sum_{a\in(\Z/p^m\Z)^\times}
        \varphi_f
        \left(
        \{\infty\}
        -
        \left\{\frac{a}{p^m}\right\}
        \right)
        \Big|_{(X,Y)=(0,1)}
        \sigma_a
        \in
        \mathcal O
        \left[
        \Gal(\Q(\mu_{p^m})/\Q)
        \right],
\end{equation}
where $\sigma_a(\zeta_{p^m})=\zeta_{p^m}^a$. This is the normalization used in
\cite[\S2]{PollackWeston}.

\par If \(\chi\) is a Dirichlet character of conductor \(p^m\), we use
the convention $\chi(\sigma_a)=\chi(a)$ and define
\[
        \tau(\chi)
        :=
        \sum_{a\bmod p^m}
        \chi(a)\zeta_{p^m}^a.
\]
For \(\varepsilon\in\{\pm1\}\), write
\[
        \Omega_f^\varepsilon
        :=
        \begin{cases}
        \Omega_f^+,&\text{if }\quad\varepsilon=1,\\
        \Omega_f^-,&\text{if }\quad\varepsilon=-1.
        \end{cases}
\]
\noindent The interpolation formula is the following.

\begin{proposition}[Mazur--Tate interpolation]
\label{prop:MT-interpolation}
Let \(\chi\) be primitive of conductor exactly \(p^m\), where
\(m\geq1\).  Then
\begin{equation}
\label{interpolformula}
        \chi\bigl(\vartheta_m(f)\bigr)
        =
        \tau(\chi)
        \frac{\mathrm{L}(f,\overline\chi,1)}
             {\Omega_f^{\chi(-1)}}.
\end{equation}
\end{proposition}

\begin{proof}
This is \cite[Proposition~2.3]{PollackWeston}, with the above
normalizations.
\end{proof}
\par We note that \(\mathrm{L}(f,\overline\chi,1)\) itself is not in
general algebraic, whereas the normalized value
\[
        \frac{\mathrm{L}(f,\overline\chi,1)}
             {\Omega_f^{\chi(-1)}}
\]
is algebraic; the period \(\Omega_f^{\chi(-1)}\) accounts for the
transcendental part of the special value.
\subsection{Cyclotomic fields and Artin factorization}

\par We set $G_n:=\Gal(K_n/\Q)$. Since $G_n\simeq(\Z/p^{n+1}\Z)^\times$
and \(p\) is odd, recall that there is a canonical decomposition $G_n=\Delta\times\Gamma_n$
where \[\Delta
        =
        \Gal(\Q(\mu_p)/\Q)
        \simeq
        (\Z/p\Z)^\times\]
and $\Gamma_n\simeq\Z/p^n\Z$.

\par We first describe the Artin factorization which connects these fields
with twisted modular \(L\)-functions.  Let \(\rho_f\) be the
\(\ell\)-adic Galois representation attached to \(f\), for some
\(\ell\neq p\), and define
\[
        \mathrm{L}_{K_n}(f,s)
        :=
        L(\rho_f|_{G_{K_n}},s).
\]
If \(\chi\) is a character of \(G_n\), we write
\[
        \mathrm{L}(f,\chi,s)
        :=
        L(\rho_f\otimes\chi,s).
\]
We state the Artin formalism which expresses the base-change $L$-function over $K_n$ as the product of the Dirichlet twists of $f$ corresponding to the characters of $G_n$.

\begin{proposition}[Artin formalism]
\label{prop:artin-formalism-cyclotomic}
For every \(n\geq0\),
\begin{equation}
\label{eq:Artin-cyclotomic-modular}
        \mathrm{L}_{K_n}(f,s)
        =
        \prod_{\chi\in\widehat{G_n}}
        \mathrm{L}(f,\chi,s).
\end{equation}
\end{proposition}

\begin{proof}
The induction property for Artin \(L\)-functions gives
\[
        L(\rho_f|_{G_{K_n}},s)
        =
        L
        \left(
        \rho_f\otimes
        \mathrm{Ind}_{G_{K_n}}^{G_{\Q}}\mathbf1,
        s
        \right).
\]
Since \(K_n/\Q\) is finite abelian, the induced representation is its
regular representation and decomposes as
\[
        \mathrm{Ind}_{G_{K_n}}^{G_{\Q}}\mathbf1
        \simeq
        \bigoplus_{\chi\in\widehat{G_n}}\chi.
\]
Multiplicativity of \(L\)-functions with respect to direct sums gives
\eqref{eq:Artin-cyclotomic-modular}.  See also
\cite[Proposition~3.8(ii), p.~530]{DeligneConstantes}.
\end{proof}

\par For \(n\geq1\), let $\mathcal C_{n+1}$
denote the set of primitive Dirichlet characters of conductor exactly
\(p^{n+1}\). These are precisely the characters of \(G_n\) which do
not factor through \(G_{n-1}\).

\begin{proposition}
\label{prop:successive-base-change-factorization}
For every \(n\geq1\),
\begin{equation}
\label{eq:primitive-base-change-factorization-s1}
        \mathrm{L}_{K_n}(f,1)
        =
        \mathrm{L}_{K_{n-1}}(f,1)
        \prod_{\chi\in\mathcal C_{n+1}}
        \mathrm{L}(f,\chi,1).
\end{equation}
\end{proposition}

\begin{proof}
In \eqref{eq:Artin-cyclotomic-modular}, partition
\(\widehat{G_n}\) into the characters which factor through \(G_{n-1}\)
and those which do not.  The first set is naturally
\(\widehat{G_{n-1}}\), while the second is exactly
\(\mathcal C_{n+1}\).
\end{proof}
\par Let $\omega:\Delta\to\Z_p^\times$ denote the Teichm\"uller character.
Since $G_n=\Delta\times\Gamma_n$, every element of $\mathcal O[G_n]$
is an $\mathcal O$-linear combination of pairs $(\delta,\gamma)$, with
$\delta\in\Delta$ and $\gamma\in\Gamma_n$. For each
$0\leq i\leq p-2$, the character $\omega^i$ extends
$\mathcal O$-linearly to an $\mathcal O$-algebra homomorphism
\[
        \omega^i:\mathcal O[G_n]\longrightarrow\mathcal O[\Gamma_n],
\]
defined on group elements by
\[
        \omega^i(\delta,\gamma)=\omega^i(\delta)\gamma.
\]
Thus, if
\[
        x=\sum_{\delta,\gamma}
        a_{\delta,\gamma}(\delta,\gamma)\in\mathcal O[G_n],
\]
then
\[
        \omega^i(x)
        =
        \sum_{\delta,\gamma}
        a_{\delta,\gamma}\omega^i(\delta)\gamma.
\]

\par Since \(\vartheta_{n+1}(f)\) has conductor \(p^{n+1}\), set
\begin{equation}
\label{thetaeqn}
        \theta_{n,i}(f)
        :=
        \omega^i\bigl(\vartheta_{n+1}(f)\bigr)
        \in
        \mathcal O[\Gamma_n].
\end{equation}
\par For \(n\geq1\), let
\[
        \pi_n^{n-1}:
        \mathcal O[\Gamma_n]
        \longrightarrow
        \mathcal O[\Gamma_{n-1}]
\]
be induced by the natural quotient
\(\Gamma_n\twoheadrightarrow\Gamma_{n-1}\).  We also use the
corestriction map
\[
        \operatorname{cor}_{n-1}^n:
        \mathcal O[\Gamma_{n-1}]
        \longrightarrow
        \mathcal O[\Gamma_n].
\]
If \(g\in\Gamma_{n-1}\) and \(\widetilde g\) is any lift to
\(\Gamma_n\), then
\[
        \operatorname{cor}_{n-1}^n(g)
        =
        \sum_{h\in\ker(\Gamma_n\to\Gamma_{n-1})}
        \widetilde g h.
\]

\par Pollack--Weston's three-term distribution relation
\cite[Proposition~2.5]{PollackWeston}, in the present
trivial-nebentype normalization, is
\begin{equation}
\label{eq:MT-distribution}
        \pi_{n+1}^{n}
        \bigl(\theta_{n+1,i}(f)\bigr)
        =
        a_p\theta_{n,i}(f)
        -
        p^{k-2}
        \operatorname{cor}_{n-1}^{n}
        \bigl(\theta_{n-1,i}(f)\bigr)
\end{equation}
for $n\geq 1$. Let
\[
        \operatorname{aug}_n:
        \mathcal O[\Gamma_n]\longrightarrow\mathcal O,
\]
be the map defined by \[\sum_{\gamma}c_\gamma\gamma
        \longmapsto
        \sum_{\gamma}c_\gamma,\]
and set
\begin{equation}
\label{eq:def-Ani}
        A_{n,i}
        :=
        \operatorname{aug}_n(\theta_{n,i}(f))
        \quad\text{and}\quad
        A_n(f)
        :=
        \prod_{i=0}^{p-2}A_{n,i}.
\end{equation}

\par Applying augmentation to
\eqref{eq:MT-distribution} gives an important recurrence.  The
augmentation of a corestriction is multiplied by the order \(p\) of
the kernel, so
\begin{equation}
\label{eq:augmentation-recursion}
        A_{n+1,i}
        =
        a_pA_{n,i}
        -
        p^{k-1}A_{n-1,i},
        \qquad n\geq1.
\end{equation}

\begin{lemma}
\label{lem:Ani-integral-formula}
Let \(n\geq0\) and \(0\leq i\leq p-2\), and put
\[
        \varepsilon_i:=\omega^i(-1)\in\{\pm1\}.
\]
Then
\begin{equation}
\label{eq:Ani-modular-symbol-formula}
        A_{n,i}
        =
        \sum_{a\in(\Z/p^{n+1}\Z)^\times}
        \omega^i(a\bmod p)\,
        \varphi_f^{\varepsilon_i}
        \left(
        \{\infty\}
        -
        \left\{\frac{a}{p^{n+1}}\right\}
        \right)\Big|_{(X,Y)=(0,1)}.
\end{equation}
Equivalently,
\begin{equation}
\label{eq:Ani-integral-formula}
        A_{n,i}
        =
        \frac{2\pi i}{\Omega_f^{\varepsilon_i}}
        \sum_{a\in(\Z/p^{n+1}\Z)^\times}
        \omega^i(a\bmod p)
        \int_{a/p^{n+1}}^{i\infty}f(z)\,dz.
\end{equation}
\end{lemma}
\begin{proof}
By \eqref{thetaeqn} and \eqref{eq:def-Ani}, together with the
definition \eqref{eq:MT-element} of the Mazur--Tate element, we have
\[
        A_{n,i}
        =
        \sum_{a\in(\Z/p^{n+1}\Z)^\times}
        \omega^i(a\bmod p)\,
        \varphi_f
        \left(
        \{\infty\}
        -
        \left\{\frac{a}{p^{n+1}}\right\}
        \right)\Big|_{(X,Y)=(0,1)}.
\]
For brevity, put
\[
        D_a
        :=
        \{\infty\}
        -
        \left\{\frac{a}{p^{n+1}}\right\}.
\]
Since
\[
        \varphi_f=\varphi_f^++\varphi_f^-,
\]
it suffices to determine which of the two eigensymbols contributes to
the preceding sum.  For \(\delta\in\{\pm1\}\), set
\[
        T_\delta
        :=
        \sum_{a\in(\Z/p^{n+1}\Z)^\times}
        \omega^i(a\bmod p)\,
        \varphi_f^\delta(D_a)\Big|_{(X,Y)=(0,1)}.
\]
The set \((\Z/p^{n+1}\Z)^\times\) is stable under \(a\mapsto-a\).
Making this change of variables and using
\[
        \omega^i(-a\bmod p)
        =
        \varepsilon_i\omega^i(a\bmod p)
\]
and
\[
        \varphi_f^\delta(D_{-a})\Big|_{(X,Y)=(0,1)}
        =
        \delta\,
        \varphi_f^\delta(D_a)\Big|_{(X,Y)=(0,1)},
\]
we obtain
\[
        T_\delta
        =
        \varepsilon_i\delta\,T_\delta.
\]
Thus \(T_\delta=0\) whenever
\(\delta\neq\varepsilon_i\).  Consequently only the
\(\varepsilon_i\)-eigensymbol contributes, which proves
\eqref{eq:Ani-modular-symbol-formula}.
\par Finally,
\[
        \varphi_f^{\varepsilon_i}
        =
        \frac{\xi_f^{\varepsilon_i}}
             {\Omega_f^{\varepsilon_i}}.
\]
The same argument shows that, after summing with the weights
\(\omega^i(a\bmod p)\), the other eigensymbol component of \(\xi_f\)
vanishes. Hence
\[
\begin{aligned}
        A_{n,i}
        &=
        \frac{1}{\Omega_f^{\varepsilon_i}}
        \sum_{a\in(\Z/p^{n+1}\Z)^\times}
        \omega^i(a\bmod p)\,
        \xi_f(D_a)\Big|_{(X,Y)=(0,1)} \\
        &=
        \frac{2\pi i}{\Omega_f^{\varepsilon_i}}
        \sum_{a\in(\Z/p^{n+1}\Z)^\times}
        \omega^i(a\bmod p)
        \int_{a/p^{n+1}}^{i\infty}f(z)\,dz,
\end{aligned}
\]
by the definition of the complex modular symbol; cf. \eqref{complexmodsym}. This proves
\eqref{eq:Ani-integral-formula}.
\end{proof}

\par We also isolate the characters which are new at level \(n\).  For
\(n\geq1\), define
\begin{equation}
\label{eq:def-Sni}
        \mathcal S_{n,i}(f)
        :=
        \prod_{\substack{\psi\in\widehat{\Gamma_n}\\
                         \operatorname{ord}(\psi)=p^n}}
        \psi\bigl(\theta_{n,i}(f)\bigr).
\end{equation}
Only these exact-order characters correspond to primitive characters
of conductor \(p^{n+1}\), and hence only this product is needed for
the base-change \(L\)-value.
\par Put $d_n
        :=
        \#\mathcal C_{n+1}
        =
        p^{n-1}(p-1)^2$. Multiplication by the odd character \(\omega\) gives a
parity-reversing bijection of \(\mathcal C_{n+1}\).  Hence precisely
\(d_n/2\) characters in \(\mathcal C_{n+1}\) are even and \(d_n/2\)
are odd. Define the algebraically normalized contribution of the new
characters by
\begin{equation}
\label{eq:def-new-layer-algebraic}
        \mathscr L_n^{\operatorname{alg}}(f)
        :=
        \left(
        \prod_{\chi\in\mathcal C_{n+1}}
        \tau(\chi)
        \right)
        (\Omega_f^+)^{-d_n/2}
        (\Omega_f^-)^{-d_n/2}
        \prod_{\chi\in\mathcal C_{n+1}}
        \mathrm{L}(f,\chi,1).
\end{equation}
This quantity arises in relating the algebraic parts of $L_{K_n}(f,1)$ and $L_{K_{n+1}}(f,1)$; cf. Lemma \ref{lem:algebraic-layer-factorization} below.
\begin{lemma}
\label{lem:primitive-layer-identification}
For every \(n\geq1\),
\begin{equation}
\label{eq:primitive-layer-identification}
        \mathscr L_n^{\operatorname{alg}}(f)
        =
        \prod_{i=0}^{p-2}\mathcal S_{n,i}(f).
\end{equation}
Further, we have that $\mathscr L_n^{\operatorname{alg}}(f)\in\mathcal O$.
\end{lemma}

\begin{proof}
Every character of \(G_n=\Delta\times\Gamma_n\) has a unique
decomposition $\chi=\omega^i\psi$, $0\leq i\leq p-2$ and $\psi\in\widehat{\Gamma_n}$. Such a character has conductor exactly \(p^{n+1}\) if and only if
\(\psi\) has exact order \(p^n\).  Thus
\[
        \mathcal C_{n+1}
        =
        \{
        \omega^i\psi:
        0\leq i\leq p-2,\
        \operatorname{ord}(\psi)=p^n
        \}.
\]

For such a pair \((i,\psi)\), evaluation of
\(\theta_{n,i}(f)\) is the same as evaluation of
\(\vartheta_{n+1}(f)\) at \(\omega^i\psi\).  Therefore
\eqref{interpolformula} gives
\[
        \psi(\theta_{n,i}(f))
        =
        \tau(\omega^i\psi)
        \frac{\mathrm{L}(f,\overline{\omega^i\psi},1)}
             {\Omega_f^{(\omega^i\psi)(-1)}}.
\]
Multiplying over all \(i\) and all exact-order \(\psi\) gives
\[
        \prod_{i=0}^{p-2}\mathcal S_{n,i}(f)
        =
        \prod_{\chi\in\mathcal C_{n+1}}
        \tau(\chi)
        \frac{\mathrm{L}(f,\overline\chi,1)}
             {\Omega_f^{\chi(-1)}}.
\]
Since \(\mathcal C_{n+1}\) is stable under inversion,
\[
        \prod_{\chi\in\mathcal C_{n+1}}
        \mathrm{L}(f,\overline\chi,1)
        =
        \prod_{\chi\in\mathcal C_{n+1}}
        \mathrm{L}(f,\chi,1).
\]
The parity count noted above now gives
\eqref{eq:primitive-layer-identification}.

To see integrality directly, choose a generator \(\gamma_n\) of
\(\Gamma_n\) and write
\[
        \theta_{n,i}(f)
        =
        \sum_{j=0}^{p^n-1}c_{j,n,i}\gamma_n^j
        \quad\text{and}\quad
        F_{n,i}(X)
        :=
        \sum_{j=0}^{p^n-1}c_{j,n,i}X^j
        \in\mathcal O[X].
\]
Then
\[
        \mathcal S_{n,i}(f)
        =
        \prod_{\operatorname{ord}(\zeta)=p^n}
        F_{n,i}(\zeta)
        =
        \operatorname{Res}
        \bigl(
        \Phi_{p^n}(X),F_{n,i}(X)
        \bigr).
\]
Since both polynomials lie in \(\mathcal O[X]\), the resultant belongs
to \(\mathcal O\).  Thus each \(\mathcal S_{n,i}(f)\), and hence their
product, lies in \(\mathcal O\).
\end{proof}

\begin{lemma}
\label{lem:primitive-layer-reduction}
For every \(n\geq1\),
\begin{equation}
\label{eq:primitive-layer-algebraic-congruence}
        \mathscr L_n^{\operatorname{alg}}(f)
        \equiv
        A_n(f)^{\varphi(p^n)}
        \pmod{p\mathcal O},
\end{equation}
(cf. \eqref{eq:def-Ani} for the definition of $A_n(f)$).
\end{lemma}

\begin{proof}
With the notation in the preceding proof,
\[
        \mathcal S_{n,i}(f)
        =
        \operatorname{Res}
        \bigl(
        \Phi_{p^n}(X),F_{n,i}(X)
        \bigr).
\]
Now
\[
        \Phi_{p^n}(X)
        =
        \frac{X^{p^n}-1}
             {X^{p^{n-1}}-1}.
\]
Over \(\F_p\),
\[
        X^{p^r}-1=(X-1)^{p^r},
\]
and therefore
\begin{equation}
\label{eq:cyclotomic-polynomial-reduction}
        \Phi_{p^n}(X)
        \equiv
        (X-1)^{p^n-p^{n-1}}
        =
        (X-1)^{\varphi(p^n)}
        \pmod p.
\end{equation}
Reducing the resultant modulo \(p\) consequently gives
\[
\begin{aligned}
        \mathcal S_{n,i}(f)
        &\equiv
        \operatorname{Res}
        \left(
        (X-1)^{\varphi(p^n)},
        F_{n,i}(X)
        \right)
        \pmod{p\mathcal O}\\
        &=
        F_{n,i}(1)^{\varphi(p^n)}
        \pmod{p\mathcal O}.
\end{aligned}
\]
But evaluation at \(X=1\) is augmentation, so
\[
        F_{n,i}(1)=A_{n,i}.
\]
Hence
\[
        \mathcal S_{n,i}(f)
        \equiv
        A_{n,i}^{\varphi(p^n)}
        \pmod{p\mathcal O}.
\]
Multiplication over \(0\leq i\leq p-2\), followed by
Lemma~\ref{lem:primitive-layer-identification}, gives
\eqref{eq:primitive-layer-algebraic-congruence}.
\end{proof}

\subsection{Algebraic base-change \(L\)-values}
\label{subsec:alg-L-stability}

\par We now study the algebraic $L$-values with respect to base change. Let $\mathcal X_n:=\widehat{G_n}$ and for \(\chi\in\mathcal X_n\), let \(\chi^\circ\) denote the primitive
Dirichlet character inducing \(\chi\).  Thus
\[
        \mathrm{L}(f,\chi,s)=\mathrm{L}(f,\chi^\circ,s).
\]
For the trivial character, put $\tau(\mathbf1):=1$. Finally, write
\[
        \mathcal X_n^\pm
        :=
        \{\chi\in\mathcal X_n:\chi(-1)=\pm1\}.
\]

\begin{definition}
\label{def:algebraic-base-change}
The algebraic special value of \(f\) over \(K_n\) is
\begin{equation}
\label{eq:def-algebraic-base-change}
        \mathrm{L}_{K_n}^{\operatorname{alg}}(f,1)
        :=
        \left(
        \prod_{\chi\in\mathcal X_n}
        \tau(\chi^\circ)
        \right)
        (\Omega_f^+)^{-\#\mathcal X_n^+}
        (\Omega_f^-)^{-\#\mathcal X_n^-}
        \mathrm{L}_{K_n}(f,1).
\end{equation}
\end{definition}
\par The normalization is compatible with passage from one cyclotomic
layer to the next.

\begin{lemma}
\label{lem:algebraic-layer-factorization}
For every \(n\geq1\),
\begin{equation}
\label{eq:algebraic-layer-factorization}
        \mathrm{L}_{K_n}^{\operatorname{alg}}(f,1)
        =
        \mathrm{L}_{K_{n-1}}^{\operatorname{alg}}(f,1)
        \mathscr L_n^{\operatorname{alg}}(f).
\end{equation}
\end{lemma}

\begin{proof}
The characters of \(G_n\) are the disjoint union of those which factor
through \(G_{n-1}\) and the characters in
\(\mathcal C_{n+1}\).  Therefore both the Gauss product and the period
product in \eqref{eq:def-algebraic-base-change} split into the old
factor and the primitive new factor.  Proposition~
\ref{prop:successive-base-change-factorization} gives the same
factorization for the complex \(L\)-value.  Combining the three
factorizations gives
\eqref{eq:algebraic-layer-factorization}.
\end{proof}
\begin{proposition}
\label{prop:algebraic-L-integrality}
For every \(n\geq0\), one has that $\mathrm{L}_{K_n}^{\operatorname{alg}}(f,1)
        \in\mathcal O$.
\end{proposition}

\begin{proof}
We first prove the result for $n=0$. By \eqref{eq:def-algebraic-base-change} and the Artin factorization
\eqref{eq:Artin-cyclotomic-modular}, we have
\[
        \mathrm{L}_{K_0}^{\operatorname{alg}}(f,1)
        =
        \frac{\mathrm{L}(f,1)}{\Omega_f^+}
        \prod_{i=1}^{p-2}
        \tau(\omega^i)
        \frac{\mathrm{L}(f,\overline{\omega^i},1)}
             {\Omega_f^{\omega^i(-1)}}.
\]
We recall from \eqref{eq:def-Ani} that \[A_{n,i}
        :=
        \operatorname{aug}_n(\theta_{n,i}(f)).\] For \(1\leq i\leq p-2\), the
character \(\omega^i\) is a nontrivial primitive character of conductor
\(p\). Since \(\Gamma_0\) is trivial, we have that
\[
        \theta_{0,i}(f)
        =
        \omega^i(\vartheta_1(f))
        =
        A_{0,i}.
\]
Thus, applying the Mazur--Tate interpolation formula
\eqref{interpolformula} with \(\chi=\omega^i\), we obtain
\begin{equation}
\label{eq:A0i-interpolation}
        A_{0,i}
        =
        \tau(\omega^i)
        \frac{\mathrm{L}(f,\overline{\omega^i},1)}
             {\Omega_f^{\omega^i(-1)}}
        \in\mathcal O.
\end{equation}

It remains to treat the trivial character. Put
\(D_0=\{\infty\}-\{0\}\). By the choice of cohomological period,
\(\varphi_f^+(D_0)\in V_{k-2}(\mathcal O)\). Specializing at
\((X,Y)=(0,1)\) and using the definition of \(\xi_f\) gives, up to the
fixed sign convention for the modular symbol,
\[
        \varphi_f^+(D_0)\big|_{(0,1)}
        =
        \pm\frac{\mathrm{L}(f,1)}{\Omega_f^+}.
\]
Hence
\begin{equation}
\label{eq:trivial-algebraic-integrality}
        \frac{\mathrm{L}(f,1)}{\Omega_f^+}
        \in\mathcal O.
\end{equation}
Together with \eqref{eq:A0i-interpolation}, this proves
\(\mathrm{L}_{K_0}^{\operatorname{alg}}(f,1)\in\mathcal O\).

For \(n\geq1\), Lemma~\ref{lem:primitive-layer-identification} gives
\(\mathscr L_n^{\operatorname{alg}}(f)\in\mathcal O\). The result
therefore follows inductively from
\eqref{eq:algebraic-layer-factorization}.
\end{proof}

\begin{proposition}
\label{prop:general-algebraic-L-congruence}
For every \(n\geq1\),
\begin{equation}
\label{eq:general-algebraic-L-congruence}
        \mathrm{L}_{K_n}^{\operatorname{alg}}(f,1)
        \equiv
        A_n(f)^{\varphi(p^n)}
        \mathrm{L}_{K_{n-1}}^{\operatorname{alg}}(f,1)
        \pmod{p\mathcal O}.
\end{equation}
\end{proposition}

\begin{proof}
By Lemma~\ref{lem:primitive-layer-reduction},
\[
        \mathscr L_n^{\operatorname{alg}}(f)
        \equiv
        A_n(f)^{\varphi(p^n)}
        \pmod{p\mathcal O}.
\]
Multiplying by
\(\mathrm{L}_{K_{n-1}}^{\operatorname{alg}}(f,1)\in\mathcal O\) and
using \eqref{eq:algebraic-layer-factorization} proves the assertion.
\end{proof}

\subsection{Relation with $\mathrm{L}_{K_0}^{\operatorname{alg}}(f,1)$}

\par We now relate \(A_n(f)\) to $\mathrm{L}_{K_0}^{\operatorname{alg}}(f,1)$.
This is the step which eventually allows the augmentation factor
$A_n(f)^{\varphi(p^n)}$ to be eliminated when \(\kappa=\F_p\).

\begin{lemma}
\label{lem:A0-bottom}
One has
\begin{equation}
\label{eq:A0-exact-bottom-value}
        A_0(f)
        =
        \bigl(1-a_p+p^{k-2}\bigr)
        \mathrm{L}_{K_0}^{\operatorname{alg}}(f,1).
\end{equation}
Consequently,
\begin{equation}
\label{eq:A0-mod-m-bottom-value}
        A_0(f)
        \equiv
        (1-a_p)
        \mathrm{L}_{K_0}^{\operatorname{alg}}(f,1)
        \pmod{\mathfrak m}.
\end{equation}
\end{lemma}

\begin{proof}
For \(1\leq i\leq p-2\), equation
\eqref{eq:A0i-interpolation} identifies \(A_{0,i}\) with the
normalized twisted \(L\)-value corresponding to the nontrivial
character \(\omega^i\). Since
\[
        \mathrm{L}_{K_0}^{\operatorname{alg}}(f,1)
        =
        \frac{\mathrm{L}(f,1)}{\Omega_f^+}
        \prod_{i=1}^{p-2}A_{0,i}
\]
and \(A_0(f)=\prod_{i=0}^{p-2}A_{0,i}\), the only factor in
\(A_0(f)\) not determined by \eqref{eq:A0i-interpolation} is the
trivial-character term \(A_{0,0}\). 
\par It therefore remains to compute
\(A_{0,0}\). Invoking Lemma \ref{lem:Ani-integral-formula}, we have that
\[
        A_{0,0}
        =
        \frac{2\pi i}{\Omega_f^+}
        \sum_{a=1}^{p-1}
        \int_{a/p}^{i\infty}f(z)\,dz.
\]
For \(y>0\), termwise integration
of the absolutely convergent Fourier series gives
\[
        2\pi i
        \int_{a/p+iy}^{i\infty}f(z)\,dz
        =
        -\sum_{r\geq1}
        \frac{a_r(f)}{r}
        e^{2\pi i ar/p}e^{-2\pi ry}.
\]
We carry out the calculation for \(y>0\) and then let \(y\to0^+\).

Summing over \(a=1,\ldots,p-1\), we may interchange the two sums.
The inner exponential sum is
\[
        \sum_{a=1}^{p-1}e^{2\pi i ar/p}
        =
        -1+p\,\mathbf 1_{p\mid r}.
\]
Consequently, if
\[
        F(y)
        :=
        \sum_{m\geq1}\frac{a_m(f)}m e^{-2\pi my},
\]
then
\begin{equation}
\label{eq:A00-first}
        A_{0,0}
        =
        \frac1{\Omega_f^+}
        \lim_{y\to0^+}
        \left(
        F(y)
        -
        p\sum_{p\mid r}
        \frac{a_r(f)}r e^{-2\pi ry}
        \right).
\end{equation}
\par For \(y>0\), set
\[
        F(y)
        :=
        \sum_{m\geq1}\frac{a_m(f)}m e^{-2\pi my}.
\]
Using the Fourier expansion
\[
        f(it)=\sum_{m\geq1}a_m(f)e^{-2\pi mt},
\]
which is absolutely convergent for \(t>0\), we may integrate term by
term on \([y,\infty)\). This gives
\[
        2\pi\int_y^\infty f(it)\,dt
        =
        \sum_{m\geq1}\frac{a_m(f)}m e^{-2\pi my}
        =
        F(y).
\]
Letting \(y\to0^+\), and using the cuspidality of \(f\), we obtain
\[
        \lim_{y\to0^+}F(y)
        =
        2\pi\int_0^\infty f(it)\,dt.
\]
The Mellin transform of \(f\) gives
\[
        \int_0^\infty f(it)t^{s-1}\,dt
        =
        (2\pi)^{-s}\Gamma(s)\mathrm{L}(f,s),
\]
and hence, at \(s=1\),
\[
        2\pi\int_0^\infty f(it)\,dt
        =
        \mathrm{L}(f,1).
\]
Therefore
\[
        \lim_{y\to0^+}F(y)=\mathrm{L}(f,1).
\]
\par It remains to
evaluate the contribution from the Fourier coefficients whose indices
are divisible by \(p\).

Writing \(r=pm\), this contribution becomes
\[
        p\sum_{p\mid r}\frac{a_r(f)}r e^{-2\pi ry}
        =
        \sum_{m\geq1}\frac{a_{pm}(f)}m e^{-2\pi pmy}.
\]
Since \(p\nmid N\) and the nebentype is trivial, the Hecke relation at
\(p\) is
\[
        a_{pm}(f)
        =
        a_pa_m(f)-p^{k-1}a_{m/p}(f),
\]
where \(a_{m/p}(f)=0\) when \(p\nmid m\). Substituting this relation,
the first term gives \(a_pF(py)\). In the second term only the
multiples \(m=pn\) occur, and hence
\[
        p^{k-1}
        \sum_{m\geq1}
        \frac{a_{m/p}(f)}m e^{-2\pi pmy}
        =
        p^{k-2}F(p^2y).
\]
Therefore
\[
        \sum_{m\geq1}\frac{a_{pm}(f)}m e^{-2\pi pmy}
        =
        a_pF(py)-p^{k-2}F(p^2y).
\]
As \(y\to0^+\), both \(F(py)\) and \(F(p^2y)\) tend to
\(\mathrm{L}(f,1)\). Hence the \(p\)-divisible contribution in
\eqref{eq:A00-first} tends to
\((a_p-p^{k-2})\mathrm{L}(f,1)\). Substituting this into
\eqref{eq:A00-first} gives
\begin{equation}
\label{eq:A00-bottom}
        A_{0,0}
        =
        \bigl(1-a_p+p^{k-2}\bigr)
        \frac{\mathrm{L}(f,1)}{\Omega_f^+}.
\end{equation}

Multiplying \eqref{eq:A00-bottom} by the identities
\eqref{eq:A0i-interpolation} for \(1\leq i\leq p-2\) gives
\eqref{eq:A0-exact-bottom-value}. Finally, since \(k>2\), we have
\(p^{k-2}\in\mathfrak m\), and reducing
\eqref{eq:A0-exact-bottom-value} modulo \(\mathfrak m\) gives
\eqref{eq:A0-mod-m-bottom-value}.
\end{proof}

For \(1\leq i\leq p-2\) and \(r\in\Z\), define
\[
        S_i(r)
        :=
        \sum_{a\in(\Z/p^2\Z)^\times}
        \omega^i(a\bmod p)e^{2\pi i ar/p^2}.
\]
Thus \(S_i(r)\) is the exponential sum which arises when the
character \(\omega^i\), originally defined modulo \(p\), is pulled
back along the reduction map
\[
        (\Z/p^2\Z)^\times\longrightarrow(\Z/p\Z)^\times.
\]
As usual, when \(\omega^i\) is viewed as a Dirichlet character on
\(\Z\), we extend it by \(0\) on integers divisible by \(p\).

\begin{lemma}
\label{lem:inflated-gauss-sum}
For \(1\leq i\leq p-2\), one has
\begin{equation}
\label{eq:p2-inflated-gauss-sum}
        S_i(r)
        =
        \begin{cases}
        0,&p\nmid r,\\[3pt]
        p\,\tau(\omega^i)\overline{\omega^i}(t),&r=pt.
        \end{cases}
\end{equation}
Here, in the second case,
\(\overline{\omega^i}(t)=0\) if \(p\mid t\).
\end{lemma}

\begin{proof}
Every unit modulo \(p^2\) may be written uniquely in the form
\[
        a=b+pj,
        \qquad
        1\leq b\leq p-1,\qquad
        0\leq j\leq p-1.
\]
Since \(\omega^i\) is pulled back from modulus \(p\), we have
\(\omega^i(b+pj)=\omega^i(b)\). Hence
\[
\begin{aligned}
        S_i(r)
        &=
        \sum_{b=1}^{p-1}
        \sum_{j=0}^{p-1}
        \omega^i(b)
        e^{2\pi i(b+pj)r/p^2}                                    \\
        &=
        \sum_{b=1}^{p-1}
        \omega^i(b)e^{2\pi i br/p^2}
        \sum_{j=0}^{p-1}e^{2\pi i jr/p}.
\end{aligned}
\]
The inner sum is geometric. If \(p\nmid r\), then
\(e^{2\pi i r/p}\neq1\), and therefore
\[
        \sum_{j=0}^{p-1}e^{2\pi i jr/p}
        =
        \frac{1-e^{2\pi i r}}{1-e^{2\pi i r/p}}
        =
        0.
\]
Thus \(S_i(r)=0\) whenever \(p\nmid r\).
\par Suppose now that \(r=pt\). Then the inner sum is equal to \(p\), so
\[
        S_i(pt)
        =
        p\sum_{b=1}^{p-1}
        \omega^i(b)e^{2\pi i bt/p}.
\]
If \(p\nmid t\), multiplication by \(t\) permutes
\((\Z/p\Z)^\times\). Making the change of variables
\(c\equiv bt\pmod p\), we obtain
\[
\begin{aligned}
        \sum_{b=1}^{p-1}
        \omega^i(b)e^{2\pi i bt/p}
        &=
        \sum_{c=1}^{p-1}
        \omega^i(ct^{-1})e^{2\pi i c/p}                         \\
        &=
        \omega^i(t^{-1})
        \sum_{c=1}^{p-1}\omega^i(c)e^{2\pi i c/p}                \\
        &=
        \overline{\omega^i}(t)\tau(\omega^i).
\end{aligned}
\]
Here we used
\[
        \omega^i(t^{-1})
        =
        \omega^i(t)^{-1}
        =
        \overline{\omega^i}(t).
\]
If \(p\mid t\), then
\[
        e^{2\pi i bt/p}=1
\]
for every \(b\), and hence
\[
        \sum_{b=1}^{p-1}
        \omega^i(b)e^{2\pi i bt/p}
        =
        \sum_{b=1}^{p-1}\omega^i(b)
        =
        0,
\]
since \(\omega^i\) is nontrivial. This agrees with the convention
\(\overline{\omega^i}(t)=0\) when \(p\mid t\). Therefore
\[
        S_i(pt)
        =
        p\,\tau(\omega^i)\overline{\omega^i}(t),
\]
which proves \eqref{eq:p2-inflated-gauss-sum}.
\end{proof}

For integers \(q\geq1\) and \(r\in\Z\), recall that the
\emph{Ramanujan sum} is defined by
\begin{equation}\label{ramsum}c_q(r)
        :=
        \sum_{a\in(\Z/q\Z)^\times}
        e^{2\pi i ar/q}.
\end{equation}
We shall use the following special case for \(q=p^2\).

\begin{lemma}
\label{lem:Ramanujan-p2}
For every \(r\in\Z\), one has
\begin{equation}
\label{eq:ramanujan-p2}
        c_{p^2}(r)
        =
        \begin{cases}
        0,&p\nmid r,\\
        -p,&p\mid r,\ p^2\nmid r,\\
        p(p-1),&p^2\mid r.
        \end{cases}
\end{equation}
\end{lemma}

\begin{proof}
This follows immediately from the standard formula for Ramanujan
sums; see \cite[\S 8.3, Theorem~8.6]{Apostol}.
\end{proof}

\begin{lemma}
\label{lem:A1-A0}
For every \(0\leq i\leq p-2\),
\begin{equation}
\label{eq:A1-A0}
        A_{1,i}
        \equiv
        a_pA_{0,i}
        \pmod{\mathfrak m}.
\end{equation}
For \(i\neq0\), one in fact has the exact identity
\[
        A_{1,i}=a_pA_{0,i}.
\]
\end{lemma}

\begin{proof}
We first treat the case \(1\leq i\leq p-2\). Set $\chi:=\omega^i$ and $\varepsilon:=\chi(-1)$. The character \(\chi\) is primitive of conductor \(p\). By Lemma \ref{lem:Ani-integral-formula},
\begin{equation}
\label{eq:A1i-integral}
        A_{1,i}
        =
        \frac{2\pi i}{\Omega_f^\varepsilon}
        \sum_{a\in(\Z/p^2\Z)^\times}
        \chi(a)
        \int_{a/p^2}^{i\infty}f(z)\,dz.
\end{equation}

We now evaluate the integrals in \eqref{eq:A1i-integral} using the
Fourier expansion of \(f\).  Since the points \(a/p^2\) lie on the
boundary of the upper half-plane, we first move them a distance
\(y>0\) into the upper half-plane.  For \(y>0\),
\[
        2\pi i
        \int_{a/p^2+iy}^{i\infty}f(z)\,dz
        =
        -\sum_{r\geq1}
        \frac{a_r(f)}{r}
        e^{2\pi i ar/p^2}e^{-2\pi ry}.
\]
The series on the right is absolutely convergent for every \(y>0\).
Since the sum over \(a\in(\Z/p^2\Z)^\times\) is finite, we may
therefore interchange the sum over \(a\) with the Fourier series.
Taking the limit \(y\to0^+\) only after this rearrangement gives
\[
\begin{aligned}
        A_{1,i}
        &=
        -\frac1{\Omega_f^\varepsilon}
        \lim_{y\to0^+}
        \sum_{a\in(\Z/p^2\Z)^\times}
        \chi(a)
        \sum_{r\geq1}
        \frac{a_r(f)}{r}
        e^{2\pi i ar/p^2}e^{-2\pi ry}             \\
        &=
        -\frac1{\Omega_f^\varepsilon}
        \lim_{y\to0^+}
        \sum_{r\geq1}
        \frac{a_r(f)}{r}e^{-2\pi ry}
        \sum_{a\in(\Z/p^2\Z)^\times}
        \chi(a)e^{2\pi i ar/p^2}.
\end{aligned}
\]

The factor \(p\) in \eqref{eq:p2-inflated-gauss-sum} cancels the factor
\(p\) in the denominator \(r=pt\) of the Fourier integral. Thus, with
the same modular-symbol and Gauss-sum conventions as in
\eqref{eq:A0i-interpolation}, the Abel-regularized Fourier expression
for \(A_{1,i}\) is
\[
        A_{1,i}
        =
        \frac{\tau(\chi)}{\Omega_f^\varepsilon}
        \lim_{y\to0^+}
        \sum_{t\geq1}
        \frac{a_{pt}(f)}{t}
        \overline{\chi}(t)e^{-2\pi pty}.
\]
On the other hand, \eqref{eq:A0i-interpolation} gives
\[
        A_{0,i}
        =
        \frac{\tau(\chi)}{\Omega_f^\varepsilon}
        \lim_{y\to0^+}
        \sum_{t\geq1}
        \frac{a_t(f)}{t}
        \overline{\chi}(t)e^{-2\pi ty}.
\]
We now use the Hecke relation at \(p\nmid N\):
\[
        a_{pt}(f)
        =
        a_pa_t(f)-p^{k-1}a_{t/p}(f),
\]
where \(a_{t/p}(f)=0\) if \(p\nmid t\). Substituting this into the
expression for \(A_{1,i}\), we obtain
\[
\begin{aligned}
        A_{1,i}
        &=
        \frac{\tau(\chi)}{\Omega_f^\varepsilon}
        \lim_{y\to0^+}
        \left(
        a_p
        \sum_{t\geq1}
        \frac{a_t(f)}{t}
        \overline{\chi}(t)e^{-2\pi pty}
        \right.\\
        &\hspace{45mm}\left.
        -
        p^{k-1}
        \sum_{t\geq1}
        \frac{a_{t/p}(f)}{t}
        \overline{\chi}(t)e^{-2\pi pty}
        \right).
\end{aligned}
\]
The second sum vanishes term by term. Indeed, if \(p\nmid t\), then
\(a_{t/p}(f)=0\), whereas if \(p\mid t\), then
\(\overline{\chi}(t)=0\). Hence
\[
        A_{1,i}
        =
        a_p\frac{\tau(\chi)}{\Omega_f^\varepsilon}
        \lim_{y\to0^+}
        \sum_{t\geq1}
        \frac{a_t(f)}{t}
        \overline{\chi}(t)e^{-2\pi pty}.
\]
Since \(py\to0\) as \(y\to0^+\), the last Abel limit is
\(\mathrm{L}(f,\overline{\chi},1)\). Consequently,
\[
        A_{1,i}
        =
        a_p\tau(\chi)
        \frac{\mathrm{L}(f,\overline{\chi},1)}
             {\Omega_f^\varepsilon}
        =
        a_pA_{0,i},
\]
where the last equality is \eqref{eq:A0i-interpolation}. This proves
the exact identity
\[
        A_{1,i}=a_pA_{0,i}
        \qquad
        (1\leq i\leq p-2).
\]
It remains to treat the case when \(i=0\). By Lemma \ref{lem:Ani-integral-formula}, we find that
\[
        A_{1,0}
        =
        \frac{2\pi i}{\Omega_f^+}
        \sum_{a\in(\Z/p^2\Z)^\times}
        \int_{a/p^2}^{i\infty}f(z)\,dz.
\]
To justify the Fourier expansion at the lower endpoints, we again use
Abel regularization.  For \(y>0\),
\[
        2\pi i
        \int_{a/p^2+iy}^{i\infty}f(z)\,dz
        =
        -\sum_{r\geq1}
        \frac{a_r(f)}{r}
        e^{2\pi i ar/p^2}e^{-2\pi ry},
\]
and the series on the right is absolutely convergent.  Hence, for
fixed \(y>0\), we may sum over \(a\) and interchange the finite sum
with the Fourier series.  It follows that
\[
\begin{aligned}
        A_{1,0}
        &=
        -\frac1{\Omega_f^+}
        \lim_{y\to0^+}
        \sum_{a\in(\Z/p^2\Z)^\times}
        \sum_{r\geq1}
        \frac{a_r(f)}{r}
        e^{2\pi i ar/p^2}e^{-2\pi ry} \\
        &=
        -\frac1{\Omega_f^+}
        \lim_{y\to0^+}
        \sum_{r\geq1}
        \frac{a_r(f)}{r}e^{-2\pi ry}
        \sum_{a\in(\Z/p^2\Z)^\times}
        e^{2\pi i ar/p^2}.
\end{aligned}
\]
By the definition of the Ramanujan sum \eqref{ramsum}, the inner sum is
\(c_{p^2}(r)\). Therefore
\[
        A_{1,0}
        =
        -\frac1{\Omega_f^+}
        \lim_{y\to0^+}
        \sum_{r\geq1}
        \frac{a_r(f)}r e^{-2\pi ry}c_{p^2}(r).
\]
Using \eqref{eq:ramanujan-p2}, the coefficient of a term with
\(p\mid r\), \(p^2\nmid r\), is \(p\), while the coefficient of a term
with \(p^2\mid r\) is \(-p(p-1)\). These two cases can be combined as
\[
        A_{1,0}
        =
        \frac1{\Omega_f^+}
        \lim_{y\to0^+}
        \left(
        p\sum_{p\mid r}
        \frac{a_r(f)}r e^{-2\pi ry}
        -
        p^2\sum_{p^2\mid r}
        \frac{a_r(f)}r e^{-2\pi ry}
        \right).
\]
Indeed, for \(p^2\mid r\) the coefficient on the right is
\(p-p^2=-p(p-1)\), as required. We evaluate the two terms separately. In the first sum, write \(r=pm\).
Then
\[
\begin{aligned}
        p\sum_{p\mid r}
        \frac{a_r(f)}r e^{-2\pi ry}
        &=
        \sum_{m\geq1}
        \frac{a_{pm}(f)}m e^{-2\pi pmy}.
\end{aligned}
\]
Using
\[
        a_{pm}(f)
        =
        a_pa_m(f)-p^{k-1}a_{m/p}(f),
\]
we obtain
\[
\begin{aligned}
        \sum_{m\geq1}
        \frac{a_{pm}(f)}m e^{-2\pi pmy}
        &=
        a_p
        \sum_{m\geq1}
        \frac{a_m(f)}m e^{-2\pi pmy}        \\
        &\quad
        -
        p^{k-1}
        \sum_{m\geq1}
        \frac{a_{m/p}(f)}m e^{-2\pi pmy}.
\end{aligned}
\]
The first sum is \(F(py)\). In the second sum only multiples \(m=pn\)
contribute, so
\[
\begin{aligned}
        p^{k-1}
        \sum_{m\geq1}
        \frac{a_{m/p}(f)}m e^{-2\pi pmy}
        &=
        p^{k-1}
        \sum_{n\geq1}
        \frac{a_n(f)}{pn}e^{-2\pi p^2ny} \\
        &=
        p^{k-2}F(p^2y).
\end{aligned}
\]
Therefore
\begin{equation}
\label{eq:first-p2-Abel-sum}
        \sum_{m\geq1}
        \frac{a_{pm}(f)}m e^{-2\pi pmy}
        =
        a_pF(py)-p^{k-2}F(p^2y).
\end{equation}
Since \(F(u)\to\mathrm{L}(f,1)\) as \(u\to0^+\), it follows that
\begin{equation}
\label{eq:first-p2-Abel-limit}
        \lim_{y\to0^+}
        \sum_{m\geq1}
        \frac{a_{pm}(f)}m e^{-2\pi pmy}
        =
        (a_p-p^{k-2})\mathrm{L}(f,1).
\end{equation}
For the second term, writing \(r=p^2m\) gives
\[
        p^2\sum_{p^2\mid r}
        \frac{a_r(f)}r e^{-2\pi ry}
        =
        \sum_{m\geq1}
        \frac{a_{p^2m}(f)}m e^{-2\pi p^2my}.
\]
Applying the Hecke relation to the index \(pm\), we have
\[
        a_{p^2m}(f)
        =
        a_pa_{pm}(f)-p^{k-1}a_m(f).
\]
Consequently,
\[
\begin{aligned}
        \sum_{m\geq1}
        \frac{a_{p^2m}(f)}m e^{-2\pi p^2my}
        &=
        a_p
        \sum_{m\geq1}
        \frac{a_{pm}(f)}m e^{-2\pi p^2my}       \\
        &\quad
        -
        p^{k-1}
        \sum_{m\geq1}
        \frac{a_m(f)}m e^{-2\pi p^2my}.
\end{aligned}
\]
The second sum is \(F(p^2y)\). Applying
\eqref{eq:first-p2-Abel-sum} with \(y\) replaced by \(py\) to the
first sum gives
\[
        \sum_{m\geq1}
        \frac{a_{pm}(f)}m e^{-2\pi p^2my}
        =
        a_pF(p^2y)-p^{k-2}F(p^3y).
\]
Hence
\[
\begin{aligned}
        \sum_{m\geq1}
        \frac{a_{p^2m}(f)}m e^{-2\pi p^2my}
        &=
        a_p
        \bigl(
        a_pF(p^2y)-p^{k-2}F(p^3y)
        \bigr)
        -
        p^{k-1}F(p^2y).
\end{aligned}
\]
Letting \(y\to0^+\), we find
\begin{equation}
\label{eq:second-p2-Abel-limit}
\begin{aligned}
        \lim_{y\to0^+}
        \sum_{m\geq1}
        \frac{a_{p^2m}(f)}m e^{-2\pi p^2my}
        &=
        \left(
        a_p(a_p-p^{k-2})-p^{k-1}
        \right)\mathrm{L}(f,1).
\end{aligned}
\end{equation}
Combining \eqref{eq:first-p2-Abel-limit} and
\eqref{eq:second-p2-Abel-limit}, we obtain
\[
\begin{aligned}
        A_{1,0}
        &=
        \frac1{\Omega_f^+}
        \left[
        (a_p-p^{k-2})
        -
        \bigl(
        a_p(a_p-p^{k-2})-p^{k-1}
        \bigr)
        \right]
        \mathrm{L}(f,1)                                      \\
        &=
        \frac1{\Omega_f^+}
        \left(
        a_p-a_p^2+a_pp^{k-2}-p^{k-2}+p^{k-1}
        \right)
        \mathrm{L}(f,1)                                      \\
        &=
        \left[
        a_p(1-a_p+p^{k-2})
        +
        p^{k-2}(p-1)
        \right]
        \frac{\mathrm{L}(f,1)}{\Omega_f^+}.
\end{aligned}
\]
By \eqref{eq:A00-bottom},
\[
        a_pA_{0,0}
        =
        a_p(1-a_p+p^{k-2})
        \frac{\mathrm{L}(f,1)}{\Omega_f^+}.
\]
Subtracting, we therefore obtain the exact identity
\[
        A_{1,0}-a_pA_{0,0}
        =
        p^{k-2}(p-1)
        \frac{\mathrm{L}(f,1)}{\Omega_f^+}.
\]
Since \(k>2\), we have \(p^{k-2}\in\mathfrak m\), while
\(\mathrm{L}(f,1)/\Omega_f^+\in\mathcal O\) by
\eqref{eq:trivial-algebraic-integrality}. Hence
\[
        A_{1,0}-a_pA_{0,0}\in\mathfrak m,
\]
or equivalently
\[
        A_{1,0}
        \equiv
        a_pA_{0,0}
        \pmod{\mathfrak m}.
\]
Together with the exact identity
\(A_{1,i}=a_pA_{0,i}\) for \(1\leq i\leq p-2\), this proves
\eqref{eq:A1-A0}.
\end{proof}

\begin{proposition}
\label{prop:augmentation-propagation}
Assume that \(f\) is ordinary at \(p\), so that
\(a_p\in\mathcal O^\times\). Then, for every \(n\geq0\) and every
\(0\leq i\leq p-2\),
\begin{equation}
\label{eq:Ani-propagation}
        A_{n,i}
        \equiv
        a_p^nA_{0,i}
        \pmod{\mathfrak m}.
\end{equation}
Consequently,
\begin{equation}
\label{eq:An-propagation}
        A_n(f)
        \equiv
        a_p^{n(p-1)}A_0(f)
        \pmod{\mathfrak m},
\end{equation}
and therefore
\begin{equation}
\label{eq:augmentation-bottom-general}
        A_n(f)
        \equiv
        a_p^{n(p-1)}
        (1-a_p)
        \mathrm{L}_{K_0}^{\operatorname{alg}}(f,1)
        \pmod{\mathfrak m}.
\end{equation}
\end{proposition}

\begin{proof}
The cases \(n=0,1\) follow from Lemma~\ref{lem:A1-A0}. For \(n\geq1\),
reduce \eqref{eq:augmentation-recursion} modulo \(\mathfrak m\). Since
\(k>2\), we have \(p^{k-1}\in\mathfrak m\), and therefore
\[
        A_{n+1,i}
        \equiv
        a_pA_{n,i}
        \pmod{\mathfrak m}.
\]
Induction gives \eqref{eq:Ani-propagation}. Multiplication over the
\(p-1\) Teichm\"uller components gives \eqref{eq:An-propagation}, and
Lemma~\ref{lem:A0-bottom} then gives
\eqref{eq:augmentation-bottom-general}.
\end{proof}

\par We now impose the non-anomalous condition
\begin{equation}
\label{eq:non-anomalous}
        a_p\not\equiv1\pmod{\mathfrak m}.
\end{equation}
Together with ordinarity, this implies that both \(a_p\) and \(1-a_p\)
are units in \(\mathcal O\).

\begin{corollary}
\label{cor:An-bottom-vanishing}
Assume that \(f\) is ordinary and satisfies
\eqref{eq:non-anomalous}. Then, for every \(n\geq0\),
\begin{equation}
\label{eq:An-vanishing-equivalence}
        A_n(f)\in\mathfrak m
        \quad\Longleftrightarrow\quad
        \mathrm{L}_{K_0}^{\operatorname{alg}}(f,1)\in\mathfrak m.
\end{equation}
\end{corollary}

\begin{proof}
By \eqref{eq:augmentation-bottom-general}, in \(\kappa\) one has
\[
        \overline{A_n(f)}
        =
        \overline{a_p}^{\,n(p-1)}
        (1-\overline{a_p})
        \overline{\mathrm{L}_{K_0}^{\operatorname{alg}}(f,1)},
\]
where $\bar{x}$ denotes the reduction of $x$ modulo $\mathfrak{m}$.
The first two factors are nonzero, so the left-hand side vanishes
exactly when the last factor vanishes.
\end{proof}

\begin{corollary}
\label{cor:general-unit-stability}
Assume that \(f\) is ordinary and satisfies
\eqref{eq:non-anomalous}. Then, for every \(n\geq0\),
\begin{equation}
\label{eq:general-unit-stability}
        \mathrm{L}_{K_n}^{\operatorname{alg}}(f,1)\in\mathfrak m
        \quad\Longleftrightarrow\quad
        \mathrm{L}_{K_0}^{\operatorname{alg}}(f,1)\in\mathfrak m.
\end{equation}
Equivalently,
\[
        \mathrm{L}_{K_n}^{\operatorname{alg}}(f,1)\in\mathcal O^\times
        \quad\Longleftrightarrow\quad
        \mathrm{L}_{K_0}^{\operatorname{alg}}(f,1)\in\mathcal O^\times.
\]
\end{corollary}

\begin{proof}
Suppose first that
\(\mathrm{L}_{K_0}^{\operatorname{alg}}(f,1)\in\mathfrak m\). Since
\[
        \mathrm{L}_{K_n}^{\operatorname{alg}}(f,1)
        =
        \mathrm{L}_{K_0}^{\operatorname{alg}}(f,1)
        \prod_{r=1}^{n}\mathscr L_r^{\operatorname{alg}}(f)
\]
and every new-layer factor is integral, all
\(\mathrm{L}_{K_n}^{\operatorname{alg}}(f,1)\) belong to
\(\mathfrak m\).

Conversely, suppose that $\mathrm{L}_{K_0}^{\operatorname{alg}}(f,1)$ is a unit. By
Corollary~\ref{cor:An-bottom-vanishing}, each \(A_n(f)\) is a unit.
Lemma~\ref{lem:primitive-layer-reduction} gives
\[
        \mathscr L_n^{\operatorname{alg}}(f)
        \equiv
        A_n(f)^{\varphi(p^n)}
        \pmod{p\mathcal O}.
\]
Since \(p\mathcal O\subseteq\mathfrak m\), reduction of this congruence
modulo \(\mathfrak m\) shows that its right-hand side is nonzero.
Therefore \(\mathscr L_n^{\operatorname{alg}}(f)\) is a unit. The
factorization \eqref{eq:algebraic-layer-factorization} now shows
inductively that every
\(\mathrm{L}_{K_n}^{\operatorname{alg}}(f,1)\) is a unit.
\end{proof}

\subsection{The case \(\kappa=\F_p\)}

\par The situation simplifies substantially when the residue field is
exactly \(\F_p\). If \(x\in\F_p^\times\), then \(x^{p-1}=1\). Since
\(\varphi(p^n)=p^{n-1}(p-1)\), it follows that
\(x^{\varphi(p^n)}=1\) for every \(x\in\F_p^\times\). Thus the
augmentation multiplier in
\eqref{eq:general-algebraic-L-congruence} has only two possible
reductions, namely \(0\) or \(1\).

\begin{proposition}
\label{prop:new-layer-zero-one}
Assume that \(f\) is ordinary,
\(a_p\not\equiv1\pmod{\mathfrak m}\), and \(\kappa=\F_p\). Then for
every \(n\geq1\),
\begin{equation}
\label{eq:new-layer-zero-one}
        \mathscr L_n^{\operatorname{alg}}(f)
        \equiv
        \begin{cases}
        1\pmod{\mathfrak m},
        &\text{if }\mathrm{L}_{K_0}^{\operatorname{alg}}(f,1)
        \notin\mathfrak m,\\[4pt]
        0\pmod{\mathfrak m},
        &\text{if }\mathrm{L}_{K_0}^{\operatorname{alg}}(f,1)
        \in\mathfrak m.
        \end{cases}
\end{equation}
\end{proposition}

\begin{proof}
By Corollary~\ref{cor:An-bottom-vanishing},
\[
        A_n(f)\in\mathfrak m
        \quad\Longleftrightarrow\quad
        \mathrm{L}_{K_0}^{\operatorname{alg}}(f,1)\in\mathfrak m.
\]
If \(A_n(f)\in\mathfrak m\), then
Lemma~\ref{lem:primitive-layer-reduction} gives
\(\mathscr L_n^{\operatorname{alg}}(f)\equiv0\pmod{\mathfrak m}\).

If \(A_n(f)\notin\mathfrak m\), its reduction is a nonzero element of
\(\F_p\), and hence
\(A_n(f)^{\varphi(p^n)}\equiv1\pmod{\mathfrak m}\).
Lemma~\ref{lem:primitive-layer-reduction} then gives
\(\mathscr L_n^{\operatorname{alg}}(f)\equiv1\pmod{\mathfrak m}\).
\end{proof}

\par We arrive at the main result of the section.

\begin{theorem}
\label{lthm:intro-modular}
Assume that \(f\) is ordinary at \(p\) and
\(a_p(f)\not\equiv1\pmod{\mathfrak m}\). Then
\(\mathrm{L}_{K_n}^{\operatorname{alg}}(f,1)\in\mathcal O\) for every
\(n\geq0\), and for \(n\geq1\),
\[
        \mathrm{L}_{K_n}^{\operatorname{alg}}(f,1)
        \equiv
        A_n(f)^{\varphi(p^n)}
        \mathrm{L}_{K_{n-1}}^{\operatorname{alg}}(f,1)
        \pmod{p\mathcal O}.
\]
Moreover,
\[
        \mathrm{L}_{K_n}^{\operatorname{alg}}(f,1)\in\mathfrak m
        \quad\text{if and only if}\quad
        \mathrm{L}_{K_0}^{\operatorname{alg}}(f,1)\in\mathfrak m.
\]

If \(\kappa=\F_p\), then the congruence simplifies to
\[
        \mathrm{L}_{K_n}^{\operatorname{alg}}(f,1)
        \equiv
        \mathrm{L}_{K_{n-1}}^{\operatorname{alg}}(f,1)
        \equiv
        \mathrm{L}_{K_0}^{\operatorname{alg}}(f,1)
        \pmod{\mathfrak m}.
\]
When \(\mathcal O=\Z_p\), this is a congruence modulo \(p\).
\end{theorem}

\begin{proof}
The integrality assertion, the general congruence, and the stability of
divisibility by \(\mathfrak m\) follow from
Proposition~\ref{prop:algebraic-L-integrality},
Proposition~\ref{prop:general-algebraic-L-congruence}, and
Corollary~\ref{cor:general-unit-stability}, respectively.

Assume now that \(\kappa=\F_p\). If
\(\mathrm{L}_{K_0}^{\operatorname{alg}}(f,1)\notin\mathfrak m\), then
Proposition~\ref{prop:new-layer-zero-one} gives
\(\mathscr L_n^{\operatorname{alg}}(f)\equiv1\pmod{\mathfrak m}\).
Hence, by \eqref{eq:algebraic-layer-factorization},
\[
        \mathrm{L}_{K_n}^{\operatorname{alg}}(f,1)
        \equiv
        \mathrm{L}_{K_{n-1}}^{\operatorname{alg}}(f,1)
        \pmod{\mathfrak m}.
\]
If \(\mathrm{L}_{K_0}^{\operatorname{alg}}(f,1)\in\mathfrak m\), then
Corollary~\ref{cor:general-unit-stability} shows that all
\(\mathrm{L}_{K_r}^{\operatorname{alg}}(f,1)\) lie in
\(\mathfrak m\), so the same congruence holds trivially. Iterating
gives the congruence with
\(\mathrm{L}_{K_0}^{\operatorname{alg}}(f,1)\).
\end{proof}

\bibliographystyle{alpha}
\bibliography{references}
\end{document}